\documentclass{amsart}
\usepackage[utf8]{inputenc} 
\usepackage{fullpage}
\usepackage{graphicx}
\usepackage[shortlabels]{enumitem}
\usepackage[hidelinks]{hyperref}
\usepackage{cleveref} 
\usepackage[lite]{amsrefs} 
\usepackage{titlesec}
\usepackage{array}
\usepackage{boondox-cal} 
\usepackage{float} 
\let\oldcap\cap
\usepackage{mathabx} 
\let\cap\oldcap
\usepackage{tikz}
\usetikzlibrary{positioning, arrows.meta} 
\usepackage{subcaption}
\usepackage{mathtools} 

\newcommand{\TitleWithUrl}[1]{\IfEmptyBibField{doi}%
	{\IfEmptyBibField{url}{\textit{#1}}%
		{\IfEmptyBibField{eprint}{\href {\BibField{url}}{\textit{#1}}}{\textit{#1}}}%
	}%
	{\href {https://doi.org/\BibField{doi}}{\textit{#1}}}}
\renewcommand{\eprint}[1]{\IfEmptyBibField{url}{\url{#1}}%
	{\href {\BibField{url}}{#1}}}

\BibSpec{article}{%
	+{}  {\PrintAuthors}                {author}
	+{,} { \TitleWithUrl}               {title}
	+{.} { }                            {part}
	+{:} { \textit}                     {subtitle}
	+{,} { \PrintContributions}         {contribution}
	+{.} { \PrintPartials}              {partial}
	+{,} { }                            {journal}
	+{}  { \textbf}                     {volume}
	+{}  { \PrintDatePV}                {date}
	+{,} { \issuetext}                  {number}
	+{,} { \eprintpages}                {pages}
	+{,} { }                            {status}
	+{,} { available at arXiv:\eprint}        {eprint}
	+{}  { \parenthesize}               {language}
	+{}  { \PrintTranslation}           {translation}
	+{;} { \PrintReprint}               {reprint}
	+{.} { }                            {note}
	+{.} {}                             {transition}
	+{}  {\SentenceSpace \PrintReviews} {review}
}

\BibSpec{collection.article}{%
	+{}  {\PrintAuthors}                {author}
	+{,} { \TitleWithUrl}                     {title}
	+{.} { }                            {part}
	+{:} { \textit}                     {subtitle}
	+{,} { \PrintContributions}         {contribution}
	+{,} { \PrintConference}            {conference}
	+{}  {\PrintBook}                   {book}
	+{,} { }                            {booktitle}
	+{,} { \PrintEditorsA}              {editor}
	+{,} { }                            {publisher}
	+{,} { }                            {organization}
	+{,} { }                            {address}    
	+{,} { \PrintDateB}                 {date}
	+{,} { pp.~}                        {pages}
	+{,} { }                            {status}
	+{,} { available at \eprint}        {eprint}
	+{}  { \parenthesize}               {language}
	+{}  { \PrintTranslation}           {translation}
	+{;} { \PrintReprint}               {reprint}
	+{.} { }                            {note}
	+{.} {}                             {transition}
	+{}  {\SentenceSpace \PrintReviews} {review}
}
\BibSpec{book}{%
	+{}  {\PrintPrimary}                {transition}
	+{,} { \TitleWithUrl}               {title}
	+{.} { }                            {part}
	+{:} { \textit}                     {subtitle}
	+{,} { \PrintEdition}               {edition}
	+{}  { \PrintEditorsB}              {editor}
	+{,} { \PrintTranslatorsC}          {translator}
	+{,} { \PrintContributions}         {contribution}
	+{,} { }                            {series}
	+{,} { \voltext}                    {volume}
	+{,} { }                            {publisher}
	+{,} { }                            {organization}
	+{,} { }                            {address}
	+{,} { \PrintDateB}                 {date}
	+{,} { }                            {status}
	+{}  { \parenthesize}               {language}
	+{}  { \PrintTranslation}           {translation}
	+{;} { \PrintReprint}               {reprint}
	+{.} { }                            {note}
	+{.} {}                             {transition}
	+{}  {\SentenceSpace \PrintReviews} {review}
}

\BibSpec{thesis}{%
	+{}  {\PrintAuthors}                {author}
	+{.} { \PrintDate}                  {date}
	+{.} { \TitleWithUrl}               {title}
	+{:} { \textit}                     {subtitle}
	+{,} { \PrintThesisType}            {type}
	+{,} { }                            {organization}
	+{,} { }                            {address}
	+{,} { \eprint}                     {eprint}
	+{,} { }                            {status}
	+{}  { \parenthesize}               {language}
	+{}  { \PrintTranslation}           {translation}
	+{;} { \PrintReprint}               {reprint}
	+{.} { }                            {note}
	+{.} {}                             {transition}
	+{}  {\SentenceSpace \PrintReviews} {review}
} 

\title{Loop Weierstrass Representation}
\author{Thomas Raujouan \and Nick Schmitt \and Jonas Ziefle}
\date{}

\usepackage{amsmath,amssymb,amsthm}

\newcommand{\Rr}{\mathbb{R}}
\newcommand{\Zz}{\mathbb{Z}}
\newcommand{\Cc}{\mathbb{C}}
\newcommand{\Hh}{\mathbb{H}}
\newcommand{\Ss}{\mathbb{S}}
\newcommand{\Nn}{\mathbb{N}}

\newcommand{\Qq}{\mathbb{Q}}
\newcommand{\lL}{\mathcal{L}}
\newcommand{\xX}{\mathcal{X}}

\newcommand{\jJ}{\mathcal{J}}

\newcommand{\Ee}{\mathbb{E}}
\newcommand{\CP}{\Cc P}

\newcommand{\I}{\mathrm{I}}
\newcommand{\tr}{\mathop\mathrm{tr}}
\newcommand{\SL}{\mathrm{SL}}
\newcommand{\GL}{\mathrm{GL}}
\newcommand{\SU}{\mathrm{SU}}
\newcommand{\SO}{\mathrm{SO}}
\newcommand{\su}{\mathrm{su}}

\newcommand{\la}{\lambda}
\newcommand{\inv}{^{-1}}
\renewcommand{\bar}[1]{{\overline{#1}}}

\newcommand{\wW}{\mathcal{W}}

\newcommand{\vV}{\mathcal{V}}
\newcommand{\hH}{\mathcal{H}}

\newcommand{\rR}{\mathcal{R}}

\newcommand{\sS}{\mathcal{S}}

\newcommand{\qQ}{\mathcal{Q}}
\newcommand{\Xx}{\mathbb{X}}
\newcommand{\Iso}{\mathrm{Iso}}
\newcommand{\SF}{\mathrm{SF}}
\newcommand{\RP}{\Rr\mathrm{P}}

\newcommand{\scal}[1]{\left\langle #1 \right\rangle}
\newcommand{\norm}[1]{\lVert #1 \rVert}
\newcommand{\abs}[1]{\left|#1\right|}
\def\Re{\mathop{\rm Re}}

\newcommand{\Ker}{\mathop{\mathrm{Ker}}}

\newcommand{\tendsto}[1]{\mathop{\longrightarrow}_{#1}}

\newcommand{\Dd}{\mathbb{D}}

\newcommand{\zbar}{{\bar{z}}}

\newcommand{\tT}{\mathcal{T}}

\newcommand{\ord}{\mathop{\mathrm{ord}}}
\newcommand{\St}{{\mathrm{St}}}

\newcommand{\spann}{\mathop\mathrm{span}}

\newcommand{\LWR}{\mathop{\mathrm{LWR}}}

\newcommand{\T}{\mathrm{T}}

\makeatletter
\newcommand*{\rom}[1]{\expandafter\@slowromancap\romannumeral #1@}
\makeatother

\renewcommand{\matrix}[4]{
	\begin{pmatrix}
		#1 & #2 \\ #3 & #4
 	\end{pmatrix}
 }

\newcommand{\smatrix}[4]{
	\big(\begin{smallmatrix}
		#1 & #2 \\ #3 & #4
	\end{smallmatrix}\big)
}

\newcommand{\vectorr}[2]{
	\begin{pmatrix}
		#1 \\ #2 
	\end{pmatrix}
}

\newtheorem{theorem}{Theorem}[section]
\newtheorem{lemma}[theorem]{Lemma}
\newtheorem{proposition}[theorem]{Proposition}

\newtheorem{corollary}[theorem]{Corollary}

\theoremstyle{definition}
\newtheorem{remark}[theorem]{Remark}
\newtheorem{definition}[theorem]{Definition}
\newtheorem{example}[theorem]{Example}

\numberwithin{equation}{section}

\titleformat{\section}
{\normalfont\fontsize{12}{17}\sffamily\bfseries}
{\thesection}
{1em}
{}

\titleformat{\subsection}
{\normalfont\fontsize{10}{14}\sffamily\bfseries}
{\thesubsection}
{1em}
{}

\titleformat{\subsubsection}
{\normalfont\fontsize{10}{14}\sffamily}
{\thesubsubsection}
{1em}
{}

\begin{document}

\renewcommand{\numberline}[1]{#1~}
\hypertarget{toc}{}
\setcounter{tocdepth}{1}


\begin{abstract}
	We introduce the Loop Weierstrass Representation for minimal surfaces in Euclidean space and constant mean curvature 1 surfaces in hyperbolic space by applying integral system methods to the Weierstrass and Bryant representations. We unify associated families, dual surfaces and Goursat transformations under the same holomorphic data, we introduce a simple factor dressing for minimal surfaces, and we compute and classify various examples.
\end{abstract}
\maketitle


\section*{Introduction}

Minimal surfaces in Euclidean $3$-space and constant mean curvature $1$ (CMC 1) surfaces in hyperbolic $3$-space can be produced from holomorphic data by the Weierstrass and Bryant representations, respectively~\cite{weierstrass, bryant1987}. In this article, we introduce the Loop Weierstrass Representation (LWR), a single framework that allows for the construction of both types of surfaces. It recovers the Weierstrass and Bryant representations as special cases of a wider class of holomorphic representations that arise by introducing a loop  parameter (originating from associated families) and a gauge freedom in the holomorphic data. 
Consequently, all questions posed for the Weierstrass and the Bryant representations -- such as closing, end behavior, symmetries, total curvature -- can be covered in the single framework of the LWR. This unification leads to some structural insights, technical advantages and new aspects that we expose in the present paper:

\begin{itemize}
	\item The LWR produces various families of related surfaces from the same holomorphic data. It covers the classical associated family and extends it to a 4-real-parameter family (figure~\ref{fig:dual-assoc-cat}). It also covers the Goursat transformations \cite{Hertrich-Jeromin_2003} as well as the dual CMC 1 surfaces introduced by Umehara and Yamada \cite{umehara-yamada1997}. In the LWR, all these transformations arise after solving one single ODE, whether the target space is Euclidean or hyperbolic.
	\item The LWR provides a suitable framework for simple factor dressing, an integrable system transformation that produces new surfaces from old ones. It adds planar or horospherical ends to a given surface while preserving its periods (figures~\ref{fig:dressed-catenoids-R3},~\ref{fig:dressed-catenoid2},~\ref{fig:dressed-catenoid2H3} and~\ref{fig:dressed-catenoidR31plane}).
	\item In the LWR, catenoids and surfaces with similar ends can be constructed from potentials that are locally gauge equivalent to potentials with simple poles. 
	This brings the theory of Fuchsian systems into the construction of minimal and CMC~1 surfaces, a strategy that has been fruitful for the DPW representation~\cite{dpw}.
	\item An LWR potential for a surface with symmetries can be pushed down by a covering map to a potential on a simpler surface. For example, $n$-noids with Platonic symmetries can be constructed with an LWR potential on a three-punctured $\CP^1$, and the Schwarz $P$-surface on a four-punctured $\CP^1$.
\end{itemize}

\begin{figure}[h]
	\centering
	\begin{subfigure}{0.22\textwidth}
		\includegraphics[width=\linewidth]{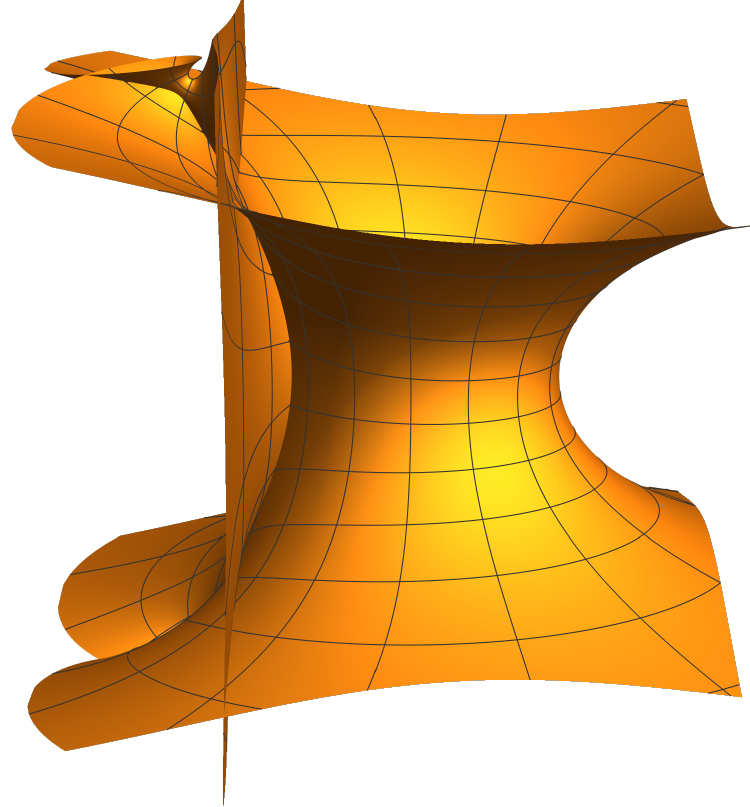}
		\caption{Side}
		\label{fig:mesh1}
	\end{subfigure}
	\hfill
	\begin{subfigure}{0.22\textwidth}
		\includegraphics[width=\linewidth]{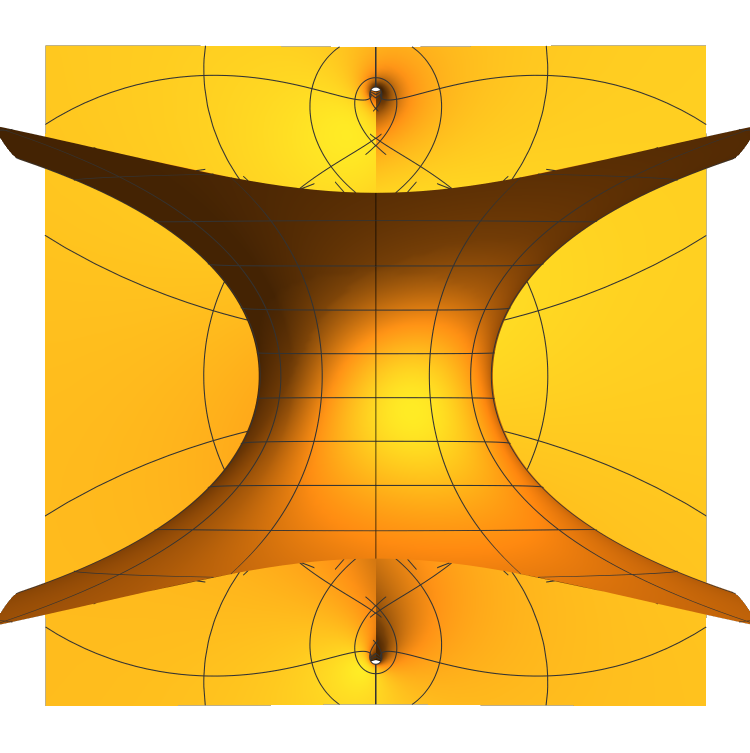}
		\caption{Front}
		\label{fig:mesh2}
	\end{subfigure}
	\hfill
	\begin{subfigure}{0.22\textwidth}
		\includegraphics[width=\linewidth]{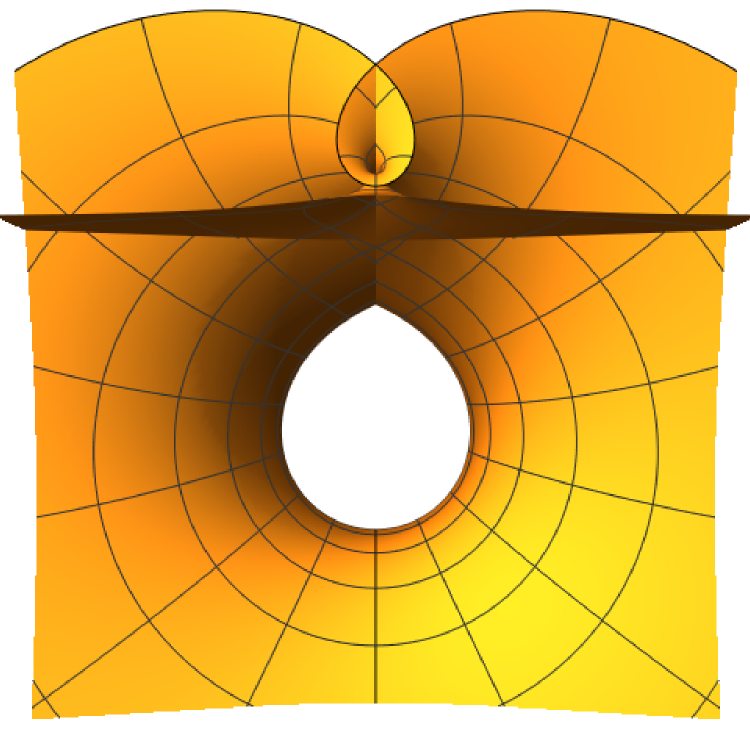}
		\caption{Top}
		\label{fig:mesh3}
	\end{subfigure}
	\hfill
	\begin{subfigure}{0.22\textwidth}
		\includegraphics[width=\linewidth]{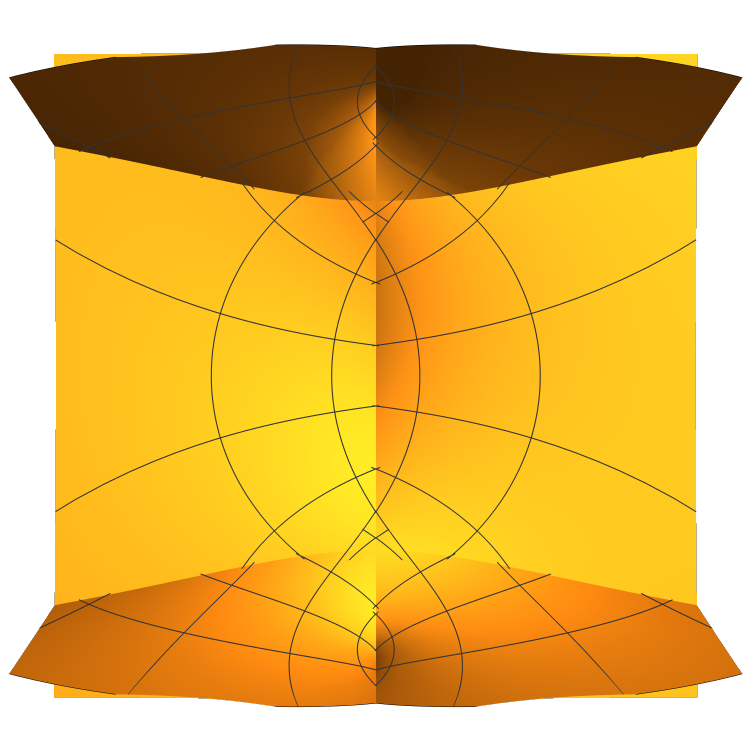}
		\caption{Back}
		\label{fig:mesh2}
	\end{subfigure}
	\caption{Simple factor dressing of the doubly-wrapped catenoid as in example~\ref{ex:dressed-catenoid1plane} with $(p,q)=(1,1)$ and $(u,\ell)=(\frac{1}{2},1)$.}
	\label{fig:dressed-catenoids-R3}
\end{figure}

\subsection*{Holomorphic Representation of Surfaces}

The representation of surfaces as immersions is based on their first and second fundamental forms, encoding the intrinsic and extrinsic geometry of the surface, respectively. The compatibility of the fundamental forms is given by the Gauss-Codazzi equations. By Bonnet's theorem, immersions can locally be constructed, provided the Gauss-Codazzi equations are satisfied.

When studying isometric constant mean curvature $H$ (CMC $H$) immersions of surfaces into spaces of constant sectional curvature $k$, the Lawson correspondence~\cite{lawson1970} states that the Gauss-Codazzi equations for two immersions are the same if the quantity $H^2+k$ is the same for both. Thereby minimal (i.e. CMC~$0$) surfaces in Euclidean space $\mathbb{E}^3$ (i.e. $k=0$) correspond to CMC~$1$ surfaces in hyperbolic space $\mathbb{H}^3$ (i.e. $k=-1$).
In the Weierstrass and Bryant representations, the correspondence becomes locally explicit in terms of an $\mathrm{sl}_2\Cc$-valued holomorphic one-form $\eta$ of rank~$1$ that encodes the first and second fundamental forms. 
Considering two different integration procedures for $\eta$ yields two different holomorphic null-curves. One represents a minimal immersion in $\mathbb{E}^3$ and is obtained by solving the ODE $d\psi = \eta$ for $\psi$ in $\mathrm{sl}_2\Cc$, the other one gives a CMC $1$ immersion in $\mathbb{H}^3$and is obtained by solving the ODE $d \Psi = \Psi \eta$ for $\Psi$ in $\SL_2\Cc$.

\subsection*{Motivation for the LWR}

In the following we explain how the LWR appears as a natural extension of the Weierstrass and Bryant representation by introducing an associate family parameter and a gauge freedom. 

Firstly, we note that scaling the Hopf differential of a CMC $1$ immersion in $\mathbb{H}^3$ by a complex parameter $\la$ and the metric by $\abs{\la}^2$ gives new solutions of the Gauss-Codazzi equation. The corresponding family of CMC $1$ immersions is called an associated family.
For the Bryant representation this amounts to scaling the potential
\begin{equation*}
	\eta \longmapsto \zeta_\lambda := \lambda \eta
\end{equation*}
and considering for the associated family of null-curves the ODE 
\begin{equation*}
	d \Psi_\lambda  = \Psi_\lambda  \zeta_\lambda.
\end{equation*}

Secondly, we observe that gauging the null-curve by an $\SL_2\Cc$-valued  holomorphic function $z \mapsto g(z)$, that is 
\begin{equation}\label{eq:Psi-lambda}
	\Psi_\la \longmapsto \Phi_\la := \Psi_\la g,
\end{equation}
 changes the potential by
\begin{equation*}
	\zeta_\la \longmapsto \xi_\la := g^{-1} \zeta_\la g +g^{-1} dg .
\end{equation*}
We thus arrive at the ODE
\begin{equation*}
	d \Phi_\lambda = \Phi_\lambda \xi_\lambda
\end{equation*}
for a potential $\xi_\lambda := \lambda \alpha +\beta$ with $\alpha$ an $\mathrm{sl}_2\Cc$-valued holomorphic one-form of rank one. 
Provided that the initial value for $\Psi_0$ has been set to the identity, then $\Phi_0 = g$ and, from \eqref{eq:Psi-lambda}, one can read off the null curve $\Psi$ as 
\begin{equation*}
	 \Psi = \Phi_1 \Phi_0^{-1}.
\end{equation*}
Furthermore, any frame of the form ${\Phi_{\la_1}\Phi_{\la_0}\inv}$ is a holomorphic null-curve and therefore produces a CMC $1$ surface.
We call this representation of CMC $1$ immersions the ``Loop Weierstrass Representation'' (LWR) and $\Phi_\lambda$ the ``LWR frame''.
The LWR possesses a gauge freedom since gauging the LWR frame by $\Phi_\lambda \mapsto \Phi_\lambda h$ leaves the null-curve unchanged. 

Null-curves $\psi$ for minimal surfaces in Euclidean space can be obtained from the LWR frame $\Phi_\lambda$ by taking a logarithmic derivative
\begin{equation*}
	\psi := (\dot{\Phi} \Phi^{-1})_{\mid \la=\la_0} ,
\end{equation*}
where the partial derivative with respect to $\lambda$ is denoted by the dot. This is reminicent of the transition from the Lie group $\SL_2\Cc$ to the Lie algebra $\mathrm{sl}_2\Cc$. As an interpretation for this formula, minimal surfaces in Euclidean space can be seen as a blow-up limit of CMC $1$ surfaces in hyperbolic space (Proposition~\ref{proposition:blow-up}).

\subsection*{Structure of the Paper}

\begin{itemize}
	\item Section~\ref{sec:lwr} provides the background for studying minimal immersions in Euclidean space and CMC $1$ immersions in hyperbolic space and introduces LWR. 
	\item In section~\ref{sec:gauging}, gauge freedom for the LWR is explained and used in order to obtain two standard forms for LWR potentials.
	\item As a follow up, section~\ref{sec:local-description} focuses on the local geometric data of surfaces in terms of the LWR, that is the first and second fundamental forms and the Gauss map.
	\item Section~\ref{sec:related-surfaces} studies how a surface changes when changing the loop parameters $(\la_0,\la_1)$ and the initial condition of the LWR frame. This yields associated families, rigid motions and as a new aspect, we see that holomorphic dressing induces Goursat transformation.
	\item The global aspects of closing a surface in the LWR and surfaces with symmetry are discussed in section~\ref{sec:symmetries}.
	\item In section~\ref{sec:examples} we describe surfaces whose first fundamental form is invariant under rotations. This provides a classification of the well-known examples of Enneper surfaces and catenoids.
	\item Section~\ref{sec:noids} is dedicated to $n$-noids and the classification of irreducible trinoids is recovered in the context of the LWR and Fuchsian systems.
	\item Finally, in section~\ref{sec:simple-factor-dressing}, we introduce simple factor dressing for minimal surfaces in Euclidean space and CMC $1$ surfaces in hyperbolic space. Note that simple factor dressing is treated without recourse to the Birkhoff splitting.
\end{itemize}

\tableofcontents


\section{The Loop Weierstrass Representation}\label{sec:lwr}

In this section, we introduce the Loop Weierstrass Representation (LWR) for constructing minimal immersions into Euclidean space and CMC $1$ immersions into hyperbolic space. LWR relies on the Lie group $\SL_2\Cc$ and its Lie algebra $\mathrm{sl}_2\Cc$, so we first recall in section~\ref{sec:ambient-spaces} how to describe Euclidean and hyperbolic spaces in a matrix model. We recall some differential geometry definitions in Section~\ref{sec:null-curves} and state the constructions of Weierstrass~\cite{weierstrass} and Bryant~\cite{bryant1987} that produce minimal and CMC $1$ surfaces out of holomorphic null curves. In Section~\ref{sec:lwr-null-curves}, we come to the LWR itself and show how it can build families of null curves from a given holomorphic connection, and can thus induce families of minimal and CMC $1$ surfaces.

\subsection{Ambient spaces}\label{sec:ambient-spaces}

Minkowski space $\Rr^{1,3}$ is identified with the real 4-dimensional vector space $\hH_2$ of 2-by-2 Hermitian matrices equipped with the Lorentzian inner product
\begin{equation}\label{eq:lorentzian-inner-product}
	\scal{x,y} := -\frac{1}{2}\tr(x\widehat{y}),\quad x,y\in\hH_2
\end{equation}
where the hat denotes the adjugate matrix, i.e.:
\begin{equation}\label{eq:adjugate}
	\widehat{x} = \matrix{d}{-b}{-c}{a} \quad \text{if}\quad x = \matrix{a}{b}{c}{d}.
\end{equation}
The Lorentzian norm is then given by the determinant:
\begin{equation}\label{eq:norm}
	\norm{x}^2 = -\det x.
\end{equation}
Under this identification, hyperbolic space is
\begin{equation*}
	\Hh^3 = \left\{x\in\hH_2\mid \det x = 1, \quad \tr x >0\right\} = \left\{FF^*\mid F\in\SL_2\Cc\right\},
\end{equation*}
and Euclidean space is identified with the tangent space of $\Hh^3$ at the identity matrix $\I$, which is nothing else but the trace-free 2-by-2 Hermitian matrices:
\begin{equation*}
	\Ee^3 = T_{\I}\Hh^3 = i\cdot \su_2.
\end{equation*}
An orientation of $\Ee^3$ is chosen so that $(x, y, x\times y)$ is positively oriented, where
\begin{equation}\label{eq:cross-product-E3}
	x\times y := -\frac{i}{2}\left[x,y\right],\quad x,y\in\Ee^3.
\end{equation} 
The group $\SL_2\Cc$ acts isometrically and transitively on $\Hh^3$ via
\begin{equation}\label{eq:isometries-H3}
	x \mapsto VxV^*
\end{equation}
where $V\in\SL_2\Cc$ and $x\in\Hh^3$. This action extends the orientation from $\Ee^3=T_{\I}\Hh^3$ to the entire $\Hh^3$.
The group $\SU_2\ltimes\Ee^3$ acts isometrically and transitively on $\Ee^3$ via
\begin{equation*}
	x \mapsto UxU\inv + T
\end{equation*}
where $U\in\SU_2$, $T\in \Ee^3$ and $x\in\Ee^3$.

\subsection{Null curves and minimal or CMC $1$ immersions}\label{sec:null-curves}

For a given function $f$ defined on a Riemann surface $\Sigma$ with local coordinate $z=x+i y$, we write  $f_x$, $f_y$, $f_z$, and $f_\bar{z}$ for the $x$, $y$, $z$ and $\bar{z}$ derivatives of $f$, respectively. In this paper, all the maps defined on a Riemann surface $\Sigma$ are assumed to be real-analytic, but not necessarily holomorphic. Let $\Xx^3$ denote Euclidean space $\Ee^3$ or hyperbolic space $\Hh^3$.

\begin{definition}
	An immersion $f\colon\Sigma\to\Xx^3$ is \textbf{conformal} if $\scal{f_z,f_z}=0$. In such case, the first fundamental form (or \textbf{metric}) reads
	\begin{equation}\label{eq:metric}
		ds^2 := \scal{df,df} = 2\scal{f_z,f_\zbar}\abs{dz}^2.
	\end{equation}

	The \textbf{unit normal} vector $N$ of a conformal immersion $f$ at a point $p = f(z_0)$ is the unit vector such that $(f_x,f_y,N)$ is a positively oriented orthogonal basis of $\T_p\Xx^3$.
	
	The \textbf{Hopf differential} $Qdz^2$ and the \textbf{mean curvature} $H$ of a conformal immersion at a point are defined via the second fundamental form:
	\begin{equation}\label{eq:definition-hopf}
		-\scal{df,dN} =: Qdz^2 + Hds^2 + \bar{Q}d\bar{z}^2.
	\end{equation}
	A conformal immersion into Euclidean space is \textbf{minimal} if $H\equiv 0$. 
	A conformal immersion into hyperbolic space is \textbf{CMC $1$} if $H\equiv 1$. 
	An \textbf{umbilic} of $f$ is a point $z_0\in\Sigma$ such that $Q(z_0)=0$. 
	If $H$ is constant and $Q$ is constantly vanishing, the immersion is \textbf{flat}.
	The \textbf{Gauss-Codazzi equations} for minimal surfaces in Euclidean space and CMC $1$ surfaces in hyperbolic space are the same. Writing $ds^2 = 4e^{2\omega}\abs{dz}^2$, they are
	\begin{equation}\label{eq:gauss-codazzi}
		\abs{Q}^2 = 4e^{2\omega}\omega_{z\zbar} ,\quad Q_\bar{z}=0.
	\end{equation}
\end{definition}

\begin{definition}\label{definition:null-curve}
	A map $\psi\colon\Sigma\to\mathrm{sl}_2\Cc$ (resp. $\Psi\colon\Sigma\to\SL_2\Cc$) is a \textbf{holomorphic null curve} if it is holomorphic with nowhere vanishing differential, and if $\det d\psi$ (resp. $\det d\Psi$) is constantly vanishing.
\end{definition}

\begin{theorem}[Weierstrass~\cite{weierstrass}]\label{theorem:weierstrass}
	Let $\psi$ be a holomorphic null curve into $\mathrm{sl}_2\Cc$ . Then $\psi + \psi^*$ is a conformal, minimal immersion into Euclidean space. Moreover, any conformal, minimal immersion can locally be obtained this way.
\end{theorem}

\begin{theorem}[Bryant~\cite{bryant1987}]\label{theorem:bryant}
	Let $\Psi$ be a holomorphic null curve into $\SL_2\Cc$. Then $\Psi^*\Psi$ is a conformal, CMC $1$ immersion into hyperbolic space. Moreover, any conformal, {CMC~1} immersion can locally be obtained this way.
\end{theorem}

\subsection{The Loop Weierstrass Representation}\label{sec:lwr-null-curves}

We first introduce the Loop Weierstrass Representation as an algorithm, and then show that it produces holomorphic null-curves (Lemma~\ref{lemma:null-curves}), and hence CMC $1$ surfaces in $\Hh^3$ and minimal surfaces in $\Ee^3$ (theorem~\ref{theorem:immersions}).

\begin{definition}\label{def:LWRpotential}
		An \textbf{LWR potential} $\xi = (\xi_\la)_{\la\in\Cc}$ is an affine linear family of meromorphic 1-forms on a Riemann surface $\Sigma$ with values in $\mathrm{sl}_2\Cc$, whose linear term is nilpotent. Locally, away from its poles, for all $\la\in\Cc$, it can be written as
		\begin{equation}\label{eq:xi=Alambda+B}
			\xi_\la = (A\la+B)dz
		\end{equation}
		where $z\in U \subset\Sigma$ is a local coordinate, $A\colon U \to\mathrm{sl}_2\Cc$ is holomorphic, nilpotent, and $B\colon U \to\mathrm{sl}_2\Cc$ is holomorphic.
		
		An \textbf{LWR frame} $\Phi=(\Phi_\la)_{\la\in\Cc}$ is a holomorphic family of maps $\Phi_\la\colon\Sigma\to\SL_2\Cc$ such that $\Phi\inv d\Phi$ is an LWR potential.

	The \textbf{Loop Weierstrass Representation} (LWR) is the following algorithm.
	\begin{enumerate}
		\item Take an LWR potential $\xi$ defined on some Riemann surface $\Sigma$.  
		\item Let $\widetilde{\Sigma}$ be the universal cover of $\Sigma$ and solve for $\Phi_\la\colon\widetilde{\Sigma}\to\SL_2\Cc$ the following initial value problem:
		\begin{equation}\label{eq:cauchy-problem}
			\begin{cases}
				\Phi\inv d\Phi = \xi,\\
				\Phi_\la(z_0) = C_\la
			\end{cases}
		\end{equation}
		where $z_0\in\widetilde{\Sigma}$ and $C_\la\in\SL_2\Cc$ is holomorphic with respect to  $\la\in\Cc$.
		\item Let $\la_0\neq\la_1\in\Cc$ and define $\psi\colon\widetilde{\Sigma}\to \mathrm{sl}_2\Cc$ and $\Psi\colon\widetilde{\Sigma}\to\SL_2\Cc$ as
		\begin{equation}\label{eq:null-curves}
			\psi := (\la_1 - \la_0)(\dot{\Phi}\Phi\inv)_{\la_0},\quad \Psi := \Phi_{\la_1}\Phi_{\la_0}\inv
		\end{equation}
		where the dot denotes the partial derivative with respect to $\la$.
		\item Define $f^\Ee\colon\widetilde{\Sigma}\to\Ee^3$ and $f^\Hh\colon\widetilde{\Sigma}\to\Hh^3$ as
		\begin{equation}\label{eq:immersions}
			f^\Ee := \psi + \psi^*,\quad f^\Hh := \Psi^*\Psi.
		\end{equation}
	\end{enumerate}
	The input data $(\Sigma, \xi, \Phi, \la_0, \la_1)$ is called \textbf{LWR data}. The ordered pair $(\la_0,\la_1)$ are the \textbf{evaluation points}. We will sometimes write $f^\Xx = \LWR(\Sigma, \xi, \Phi, \la_0, \la_1)$ to denote the immersion induced by LWR.
\end{definition}

\begin{lemma}\label{lemma:null-curves}
	Fix $\la_0\neq\la_1\in\Cc$. For any LWR frame $\Phi$, the maps $\psi$ and $\Psi$ defined in \eqref{eq:null-curves} are holomorphic null curves (possibly branched).
	Moreover, any holomorphic null curve into $\mathrm{sl}_2\Cc$ or $\SL_2\Cc$ can locally be obtained this way.
\end{lemma}
\begin{proof}
	Let $\xi$ be the LWR potential of $\Phi$. 
	We first show that $\psi$ is a holomorphic null curve. The LWR frame $\Phi$ is holomorphic, so $\psi$ is holomorphic.  Compute 
	\begin{align*}
		d(\dot{\Phi}\Phi\inv) &= d(\dot{\Phi})\Phi\inv - \dot{\Phi}\Phi\inv d{\Phi}\Phi\inv \\
		&= \frac{\partial \left(d\Phi\right)}{\partial \la}\Phi\inv - \dot{\Phi}\xi\Phi\inv\\
		&= \frac{\partial \left(\Phi \xi\right)}{\partial \la}\Phi\inv - \dot{\Phi}\xi\Phi\inv\\
		&= \Phi\dot{\xi}\Phi\inv.
	\end{align*}
	Therefore, with $\xi_\la = (A\la + B)dz$,
	\begin{equation}\label{eq:diff-null-curve-euclidean}
		d\psi = (\la_1-\la_0)\Phi_{\la_0}A\Phi_{\la_0}\inv dz.
	\end{equation}
	By definition of $\xi$, $\det A = 0$, so $\det d\psi = 0$. Similarly, $\Psi$ is holomorphic and
	\begin{equation}\label{eq:diff-null-curve-hyperbolic}
		d\Psi = (\la_1-\la_0)\Phi_{\la_1} A \Phi_{\la_0}\inv dz
	\end{equation} 
	so $\Psi$ is a holomorphic null curve.
	
	We now prove that all null curves can be obtained this way. Let $\psi\colon D\to\mathrm{sl}_2\Cc$ be a holomorphic null curve from a simply-connected neighborhood $D$ of $z_0\in\Cc$. Let $\alpha_\la = \frac{\la-\la_0}{\la_1-\la_0}$ and let $\xi_\la = \alpha_\la d\psi$. The potential $\xi$ is an LWR potential on $D$ because $\psi$ is a holomorphic null curve and $\alpha$ is linear in $\la$. 
	Define the LWR frame $\Phi$ on $D$ as the unique solution of $\Phi\inv d\Phi = \xi$ with $\Phi({z}_0) = \exp\left(\alpha\psi(z_0)\right)$. Let $X = (\la_1-\la_0)(\dot{\Phi}\Phi\inv)_{\la_0}$. A computation gives $X({z}_0) = \psi(z_0)$ because $\dot{\alpha} = (\la_1-\la_0)\inv$. Moreover, $\alpha_{\la_0}= 0$, so $X_z = \psi_z$. By the Picard-Lindelöf Theorem, $\psi = X$. 
	The case of $\Psi\colon D\to\SL_2\Cc$ is covered similarly with $\xi = \alpha \Psi\inv d\Psi$ and $\Phi({z}_0) = \exp\left(\alpha \log\Psi(z_0)\right)$, computing that $\Phi_{\la_0}=\I$ and $\Phi_{\la_1} = \Psi$.
\end{proof}

As a direct consequence of theorems~\ref{theorem:weierstrass},~\ref{theorem:bryant} and lemma~\ref{lemma:null-curves}, we have:
\begin{theorem}\label{theorem:immersions}
	Fix $\la_0\neq \la_1\in\Cc$. For any LWR frame $\Phi$, the maps $f^\Ee$ and $f^\Hh$ defined in \eqref{eq:immersions} are  minimal and CMC 1 conformal immersions, respectively.
	Moreover, any conformal, minimal (resp. CMC 1) immersion into $\Ee^3$ (resp. $\Hh^3$) can locally be obtained this way.
\end{theorem}


\section{Gauge freedom and standard forms of LWR potentials}\label{sec:gauging}

\begin{definition}\label{def:gauge}
	An \textbf{LWR gauge} is a meromorphic map $g\colon~\Sigma\to\SL_2\Cc$, independent of $\la$.
\end{definition}

\begin{lemma}\label{lemma:gauging-lwr} 
	LWR gauges act on LWR frames and potentials without changing the induced immersions.
\end{lemma}
\begin{proof}
	Let $\Phi$ be an LWR frame on $\Sigma$ with LWR potential $\xi$ and let $g\colon\Sigma\to\SL_2\Cc$ be an LWR gauge. Define
	\begin{equation*}
		\xi\cdot g := g\inv\xi g + g\inv dg.
	\end{equation*}
	It is easy to check that $\Phi\cdot g := \Phi g$ is a frame for $\xi\cdot g$. The fact that $g$ does not depend on $\la$ implies that $\xi\cdot g$ is linear in $\la$ and that the linear factor has vanishing determinant. Therefore, $\xi\cdot g$ is an LWR potential, and $\Phi g$ is an LWR frame. The $\la$-independence of $g$ also implies that the holomorphic null curves $\psi$ and $\Psi$ are unchanged after gauging. Therefore, the immersions themselves are preserved.
\end{proof}

In this section, we give two useful representatives in the gauge orbit of a given LWR potential. The first one (section~\ref{sec:weierstrass-potential}) is linear in $\la$ with no constant term and is directly related to the Weierstrass and Bryant representations. The second one (section~\ref{sec:schwarz-potential-1}) is off-diagonal and is uniquely determined by the Schwarzian derivative of a spinorial quantity that we define in section~\ref{sec:spinors}.

\textbf{Acknowledgement:} The authors wish to thank Franz Pedit for suggesting the Schwarz potential.

\subsection{Spinors}\label{sec:spinors}

\begin{definition}\label{def:spinor}
	Let $A\in\mathrm{sl}_2\Cc$ be a nilpotent matrix. A vector $x=(u,v)\in\Cc^2$ is a \textbf{spinor} for $A$ if 
	\begin{equation}\label{eq:xxperp}
		A = \matrix{-uv}{u^2}{-v^2}{uv} = \vectorr{u}{v} (-v\quad u) =: xx^\perp.
	\end{equation}
	A spinor for $A$ is uniquely determined up to a sign. We extend the definition and say that $x\colon \Dd\to \Cc^2$ is a spinor for an LWR potential $\xi$ if $xx^\perp = A$ where $\xi = (A\la + B)dz$ on some simply connected domain $\Dd\subset\Sigma$. 
\end{definition}

\begin{remark}\label{remark:xperpx}
	The notation $x^\perp$ has been chosen because for all $x=(u,v)\in\Cc^2$,
	\begin{equation*}
		x^\perp x = (-v\quad u)\vectorr{u}{v} = 0.
	\end{equation*}
\end{remark}

We give some properties of $x^\perp$ that can be directly computed:
\begin{lemma}\label{lemma:spin-properties}
	For all $x:=(u,v)\in\Cc^2$, writing $\norm{x}^2 := \abs{u}^2 + \abs{v}^2$,
	\begin{equation}\label{eq:spin-property-2}
		\scal{xx^\perp, (xx^\perp)^*} = \norm{x}^2,
	\end{equation}
	\begin{equation}\label{eq:spin-property-3}
		[xx^\perp, (xx^\perp)^*] = -\norm{x}^2(\norm{x}^2\I - 2xx^*),
	\end{equation}
	\begin{equation}\label{eq:spin-property-4}
		\norm{[xx^\perp, (xx^\perp)^*]} = \norm{y}^4,
	\end{equation}
	\begin{equation}\label{eq:spin-property-5}
		\scal{yy^*, (yy^\perp)'} = -\frac{1}{2}\norm{y}^2y^\perp y',
	\end{equation}
	\begin{equation}\label{eq:spin-property-1}
		\forall \Phi\in\SL_2\Cc,\quad \forall x \in\Cc^2,\quad(\Phi x)^\perp = x^\perp \Phi\inv.
	\end{equation}
\end{lemma}

\begin{proposition}\label{prop:q}
	Let $\xi$ be an LWR potential on $\Sigma$. Let $\Phi$ be a local LWR frame for $\xi$ and let $\la\in\Cc$. Let $y:=\Phi_\la x$ where $x$ is a spinor for $\xi$. Then the function
	\begin{equation}
		q := \det(y,y_z)
	\end{equation}
	is well-defined on $\Sigma$ and only depends on the orbit of $\xi$ under LWR gauging (not on $\Phi$, not on $\la$).
\end{proposition}
\begin{proof}
	We first show that $q$ satisfies
	\begin{equation}\label{eq:q}
		q = \det(x, Bx + x_z)
	\end{equation}
	where $\xi=(A\la + B)dz$. 
	To do so, differentiate $y=\Phi_\la x$ using $d\Phi=\Phi\xi$ and recall that $x^\perp x =0$ to get
	\begin{align*}
		dy &= \Phi\xi x + \Phi dx\\
		&= \Phi(\la A x + Bx + x_z)dz \\
		&= \Phi(\la xx^\perp x + Bx + x_z)dz\\
		&= \Phi(Bx + x_z)dz.
	\end{align*}
	Therefore, 
	\begin{equation*}
		q = \det(\Phi x, \Phi(Bx + x_z))
	\end{equation*}
	and \eqref{eq:q} follows from $\det \Phi = 1$. This shows that $q$ is well-defined on $\Sigma$ and does not depend on $\Phi$ nor does it depend on $\la$.
	
	To show that $q$ is invariant under LWR gauging, let $\tilde{\xi} := \xi\cdot g$ where $g$ is an LWR gauge. Let $\tilde{\Phi}:=\Phi g$, let $\tilde{x}$ be a spinor for $\tilde{\xi}$ and let $\tilde{y}:=\tilde{\Phi}\tilde{x}$ for some $\la\in\Cc$. Then by \eqref{eq:spin-property-1}, $\tilde{x} = g\inv x$, so $\tilde{y} = y$ and $q=\det(\tilde{y}, \tilde{y}_z)$.
\end{proof}

\begin{definition}
	An LWR potential is \textbf{degenerate} if $q$ is constantly vanishing, and \textbf{non-degenerate} otherwise.
\end{definition}

\subsection{Weierstrass potential}\label{sec:weierstrass-potential}

\begin{proposition}\label{prop:weierstrass-potential}
	Let $\xi$ be an LWR potential on $\Sigma$ and let $\Phi$ be an LWR frame for $\xi$ on a simply-connected domain $\Dd\subset\Sigma$. With $\Phi_0:=\Phi|_{\la=0}$,
	\begin{equation}\label{eq:weierstrass-potential}
		\xi\cdot\Phi_0\inv = \la\matrix{g}{-g^2}{1}{-g}\frac{qdz}{g_z}
	\end{equation}
	where 
	\begin{equation}\label{eq:defgq}
		g:=u/v,\qquad q:= uv_z-vu_z
	\end{equation}
	and
	\begin{equation*}
		(u,v):=\Phi_0 x
	\end{equation*}
	where $x$ is a spinor for $\xi$.
\end{proposition}
\begin{proof}
	Write $\xi = (A\la + B)dz$ so that $xx^\perp = A$ and let $y:=\Phi_0 x$. By \eqref{eq:spin-property-1},
	\begin{equation*}
		\Phi_0 A \Phi_0\inv = \Phi_0 x x^\perp \Phi_0\inv = yy^\perp.
	\end{equation*}
	Moreover, $d\Phi=\Phi\xi$, so
	\begin{equation*}
		\Phi_0 d(\Phi_0\inv) = -d\Phi_0 \Phi_0\inv = -\Phi_0B\Phi_0\inv dz.
	\end{equation*}
	Therefore,
	\begin{equation*}
		\xi\cdot\Phi_0\inv = \la y y^\perp dz
	\end{equation*}
	and \eqref{eq:weierstrass-potential} follows from $y=(u,v)$ and the definitions of $g$ and $q$.
\end{proof}

\begin{remark}Let $\xi_1$ be the potential given by \eqref{eq:weierstrass-potential}.
	\begin{enumerate}[1.]
		\item The functions $g$ and $q$ in $\xi_1$ do not depend on the choice of sign for the spinor $x$, and the definition of $q$ in proposition~\ref{prop:weierstrass-potential} agrees with its definition in proposition~\ref{prop:q}.
		\item Since $\Phi_0$ is only defined locally on $\Sigma$, so is the potential $\xi_1$.
		\item Since $\Phi_0$ is uniquely determined by $\xi$ up to conjugation by $\SL_2\Cc$, so is the function $g$. However, proposition~\ref{prop:q} shows that the function $q$ is uniquely determined by $\xi$.
		\item As we will see in section~\ref{sec:local-description}, the potential $\xi_1$ is directly related to the Weierstrass representation: if $\la_0 = 0$, then $g$ is the Gauss map of $f^\Ee$ and the hyperbolic Gauss map of $f^\Hh$. Moreover, if $\la_1=1$, then $qdz^2$ is the Hopf differential of $f^\Ee$ and $f^\Hh$.
	\end{enumerate}
\end{remark}

\subsection{Schwarz potential}\label{sec:schwarz-potential-1}

\begin{definition}
	Let $g\colon\Cc\to\Cc$ be a meromorphic function. The map
	\begin{equation*}
		\sS[g] = \left(\frac{g''}{2g'}\right)^2 - \left(\frac{g''}{2g'}\right)'
	\end{equation*}
	is the \textbf{Schwarzian derivative} of $g$. It is invariant under Möbius transformations of $g$. It acts as a second order operator under pre-composition: for any meromorphic $\phi\colon \Cc\to \Cc$,
	\begin{equation}\label{eq:schwarzian-deriv-equiv}
		\sS[g\circ \phi] = (\sS[g] \circ \phi)(\phi')^2 + \sS[\phi].
	\end{equation}
\end{definition}

The following lemma justifies our normalization for the Schwarzian derivative and is merely a computation.

\begin{lemma}\label{lemma:equation-schwarz}
	Let $u,v$ be two holomorphic solutions of the equation
	\begin{equation*}
		y'' =Sy 
	\end{equation*}
	where $S$ is a meromorphic function. If $u/v$ is not constant, then $S = \sS[u/v] $.
\end{lemma}

\begin{theorem}\label{theorem:schwarz-potential}
	Let $\xi=(A\la + B)dz$ be a non-degenerate LWR potential in the coordinate $z$. Then there exists an LWR gauge $k$, unique up to a sign, such that
	\begin{equation*}
		\xi\cdot k = \matrix{0}{1}{\la q +s}{0}dz
	\end{equation*}
	for some functions $q$ and $s$.
	
	Moreover, if $(u,v):=\Phi x$ where $x$ is a spinor for $A$ and $\Phi$ is an LWR frame for $\xi$ at $\la=0$, then
	\begin{equation*}
		q = uv_z-vu_z\qquad\text{and}\qquad s = \sS[u/v]|_{\la=0}.
	\end{equation*}
\end{theorem}
\begin{proof}
	We first show existence. Gauge $\xi$ with $k_1:=\Phi_0\inv$ so that $\xi_1:=\xi\cdot k_1$ is a Weierstrass potential as in proposition~\ref{prop:weierstrass-potential}. With $g:=(u/v)|_{\la=0}$, compute the successive LWR gauge transformations:
	\begin{equation*}
		k_2:=\matrix{1}{g}{0}{1},\quad \xi_2 := \xi_1\cdot k_2 = \matrix{0}{g'}{\la q/g'}{0}dz,
	\end{equation*}
	\begin{equation*}
		k_3:=\matrix{\sqrt{g'}}{0}{0}{\frac{1}{\sqrt{g'}}},\quad \xi_3:=\xi_2\cdot k_3 = \matrix{h}{1}{\la q}{-h}dz,\quad h:=\frac{g''}{2g'},
	\end{equation*}
	\begin{equation*}
		k_4 := \matrix{1}{0}{-h}{1},\quad \xi_4:=\xi_3\cdot k_4 = \matrix{0}{1}{\la q +s}{0}dz,\quad s:=\sS[g].
	\end{equation*}
	Therefore, with $k:=k_1k_2k_3k_4$, the gauged potential $\xi\cdot k$ has the expected form.
	
	We now show uniqueness. Let $r_1$ and $r_2$ be two LWR gauges such that
	\begin{equation*}
		\eta_1:=\xi\cdot r_1 = \matrix{0}{1}{\la q_1 +s_1}{0}dz\quad \text{and}\quad \eta_2:=\xi\cdot r_2 = \matrix{0}{1}{\la q_2 +s_2}{0}dz.
	\end{equation*}
	Then $\eta_2 = \eta_1\cdot r$ with $r:=r_1\inv r_2$. This implies that
	\begin{equation}\label{eq:dr}
		dr = r\eta_2 - \eta_1 r.
	\end{equation}
	With
	\begin{equation*}
		r:=\matrix{a}{b}{c}{d},
	\end{equation*}
	\eqref{eq:dr} is equivalent to
	\begin{equation*}
		\matrix{a'}{b'}{c'}{d'} = \la \matrix{bq_2}{0}{dq_2-aq_1}{bq_1} + \matrix{-c}{a-d}{ds_2-as_1}{c}.
	\end{equation*}
	But $\eta_1$ and $\eta_2$ are gauge-equivalent, so by proposition~\ref{prop:q}, $q_1=q_2=q$. Equating the $\la^1$-coefficients and recalling that $q$ is not constantly vanishing because $\xi$ is non-degenerate, $b=0$ and $a=d$. We deduce that $c=0$ and that $r=\pm \I$ because $\det r =1$. Therefore, $r_1=\pm r_2$ and $s_1=s_2$.
\end{proof}

\begin{corollary}\label{corrolary:schwarz-spinor}
	Let $\xi$ be an LWR potential. Let $\Phi$ be a frame for $\xi$, let $x$ be a spinor for $\xi$ and let $y:=\Phi x$. With $y=(u,v)$, the function $\sS[u/v]$ only depends on the orbit of $\xi$ under LWR gauging and satisfies for all $\la\in\Cc$,
	\begin{equation}\label{eq:laq+s}
		\sS[u/v] = \la q + s
	\end{equation}
	where $q$ and $s$ are as in theorem~\ref{theorem:schwarz-potential}.
\end{corollary}
\begin{proof}
	We first show that $\sS[u/v]$ does not depend on the initial condition for $\Phi$. Let $\tilde{\Phi} := C\Phi$ where $C\in\SL_2\Cc$ is holomorphic in $\la\in\Cc$ and let $\tilde{y}:=\tilde{\Phi}x$. Then  $\tilde{y} = Cy$. With $\tilde{y} = (\tilde{u}, \tilde{v})$, $\tilde{u}/\tilde{v}$ and $u/v$ differ by a $\la$-dependent M\"obius transformation. Therefore, their respective Schwarzian derivative are the same.
	
	We now show that $\sS[u/v]$ is invariant under LWR gauging. Let $g$ be an LWR gauge and let $\tilde{\Phi} := \Phi g$ and $\tilde{\xi}:=\xi\cdot g$. Then, by \eqref{eq:spin-property-1}, $\widetilde{x}:=g\inv x$ is a spinor for $\tilde{\xi}$. Let $\tilde{y} := \tilde{\Phi}\tilde{x}$. Then $\tilde{y} = \Phi x = y$, so $\sS[u/v]$ is invariant under LWR gauging.
	
	We now prove \eqref{eq:laq+s} by assuming, without loss of generality, that $\xi$ is given in its Schwarz form, as in theorem~\ref{theorem:schwarz-potential}.
	Let $y = \Phi x$ where $x$ is a spinor for $\xi$ and $\Phi$ is an LWR frame associated with $\xi$. Considering the form of $\xi$, we have $x = i\sqrt{q}(0,1)$. Writing
	\begin{equation*}
		\Phi = \matrix{a}{c}{b}{d},
	\end{equation*}
	we get
	\begin{equation*}
		y = i\sqrt{q}(c,d)
	\end{equation*}
	and
	\begin{equation*}
		\frac{u}{v} = \frac{c}{d}.
	\end{equation*}
	Now differentiate the equation $d\Phi = \Phi\xi$ componentwise in $z$ to get
	\begin{equation*}
		\begin{cases}
			c'' = a' = (\la q + s)c\\
			d'' = b' = (\la q + s)d.
		\end{cases}
	\end{equation*}
	By hypothesis, $\xi$ is non-degenerate, so by proposition~\ref{prop:q}, $\det(y,y_z)\neq 0$, therefore $c/d$ is not constant. By lemma~\ref{lemma:equation-schwarz}, $\sS[u/v] = \la q + s$.
\end{proof}

\begin{remark}
	Theorem~\ref{theorem:fundamental-forms} and corollary~\ref{corollary:gaussmap} will give a geometric interpretation of $q$ and $s$ in terms of the the induced immersion's Hopf differential and its Gauss map's Schwarzian derivative.
\end{remark}

\section{Local description in the LWR}\label{sec:local-description}

In this section, we compute the first and second fundamental forms as well as the Gauss map of immersions in terms of their LWR data. This allows us in corollary~\ref{corollary:gaussmap} to complete the geometric interpretation of the Schwarz potential introduced in section~\ref{sec:schwarz-potential-1}.

\subsection{First and second fundamental forms}

\begin{theorem}\label{theorem:fundamental-forms}
	Let $\Phi$ be an LWR frame with LWR potential $\xi$. Let $x$ be a spinor for $\xi$ and let $y:=\Phi x$.
	The metric of $f^\Xx=\LWR(\Sigma, \xi, \Phi, \la_0, \la_1)$ is
	\begin{equation}\label{eq:metric-lwr}
		ds^2 = \begin{cases}
			\abs{\la_1-\la_0}^2\norm{y_{\la_0}}^4\abs{dz}^2 & \text{ if $\Xx^3 = \Ee^3$,}\\ \abs{\la_1-\la_0}^2\norm{y_{\la_1}}^4\abs{dz}^2 & \text{ if  $\Xx^3 = \Hh^3$}.\\
		\end{cases}
	\end{equation}
	In both cases, the Hopf differential of $f^\Xx$ is
	\begin{equation}\label{eq:hopf}
		Qdz^2 = (\la_1-\la_0)qdz^2
	\end{equation}
	where $q:=\det(y,y_z)$. 
\end{theorem}
\begin{proof}
	We start with Euclidean space and compute the metric given by \eqref{eq:metric}. Consider $\psi$ defined in \eqref{eq:null-curves} and compute
	\begin{equation}\label{eq:fz-euclidean}
		f_z = \psi_z,\quad f_\zbar=(\psi_z)^*.
	\end{equation}
	By \eqref{eq:diff-null-curve-euclidean}, with $y_0 := \Phi_{\la_0}x$ (to ease the notation),
	\begin{equation}\label{eq:psiz-euclidean}
		\psi_z = (\la_1-\la_0)y_0y_0^\perp.
	\end{equation}
	The metric in $\Ee^3$ follows from \eqref{eq:spin-property-2}.
	
	To compute the Hopf differential, we first compute the normal. Write $f_x=f_z+f_\zbar$ and $f_y = i(f_z-f_\zbar)$ and recall \eqref{eq:cross-product-E3} to get
	\begin{equation}\label{eq:normal-fzfzbar}
		f_x\times f_y = -2if_z\times f_\zbar = -\left[f_z,f_\zbar\right].
	\end{equation}
	Use equations \eqref{eq:fz-euclidean}, \eqref{eq:psiz-euclidean}, \eqref{eq:spin-property-3} and \eqref{eq:spin-property-4} to get
	\begin{equation}\label{eq:N-euclidean}
		N = \I - 2\norm{y_0}^{-2}y_0y_0^*.
	\end{equation}
	To get the Hopf differential, compute
	\begin{equation*}
		f_{zz} = (\la_1-\la_0)(y_0y_0^\perp)_z.
	\end{equation*}
	Therefore, with \eqref{eq:spin-property-5},
	\begin{equation*}
		\scal{f_{zz}, N} = (\la_1-\la_0)\det(y_0, (y_0)_z).
	\end{equation*}
	By proposition~\ref{prop:q}, this gives the expected formula for the Hopf differential in $\Ee^3$.
	
	In hyperbolic space, we have
	\begin{equation}\label{eq:fz-hyperbolic}
		f_z = \Psi^*\left(\Psi_z\Psi\inv\right)\Psi,\quad f_\zbar=\Psi^*\left(\Psi_z\Psi\inv\right)^*\Psi.
	\end{equation}
	By \eqref{eq:diff-null-curve-hyperbolic} and with $y_1 := \Phi_{\la_1}x$,
	\begin{equation}\label{eq:psiz-hyperbolic}
		\Psi_z\Psi\inv = (\la_1-\la_0)y_1y_1^\perp.
	\end{equation}
	Because $\Psi$ acts as an isometry of $\Hh^3$,
	\begin{equation*}
		\scal{f_z,f_\zbar} = \scal{\Psi_z\Psi\inv, (\Psi_z\Psi\inv)^*} = \abs{\la_1-\la_0}^2\scal{y_1y_1^\perp,(y_1y_1^\perp)^*}.
	\end{equation*}
	The metric in $\Hh^3$ follows from \eqref{eq:spin-property-2}.
	
	To compute the normal in $\Hh^3$, we define $n$ with
	\begin{equation}\label{eq:orthonormal-frame-H3}
		(f_z,f_\zbar,N) =  \Psi^* (\Psi_z\Psi\inv, (\Psi_z\Psi\inv)^*, n) \Psi.
	\end{equation}
	By equations~\eqref{eq:psiz-hyperbolic}, \eqref{eq:spin-property-3} and \eqref{eq:spin-property-4}, 
	\begin{equation}\label{eq:n}
		n = \I - 2\norm{y_1}^{-2}y_1y_1^*.
	\end{equation}
	To get the Hopf differential, compute
	\begin{equation*}
		f_{zz} = (\la_1-\la_0)\Psi^*(y_1y_1^\perp)_z\Psi.
	\end{equation*}
	Then, by \eqref{eq:spin-property-5},
	\begin{equation*}
		\scal{f_{zz}, N} = (\la_1-\la_0)\det(y_1, (y_1)_z).
	\end{equation*}
	By proposition~\ref{prop:q}, this gives the expected formula for the Hopf differential in $\Hh^3$.
\end{proof}

\begin{remark}\label{remark:hopf}
	By proposition~\ref{prop:q} and \eqref{eq:q}, one can get the Hopf differential of $f$ given in \eqref{eq:hopf} without computing $\Phi$:
	\begin{equation*}
		Qdz^2 = (\la_1-\la_0)\det(x, Bx+x_z)dz^2.
	\end{equation*}
\end{remark}

\begin{remark}
	For fixed $(\la_0, \la_1)$ the immersions $f^\Ee$ and $f^\Hh$ induced by $\Phi$ are not necessarily isometric, but they have the same Hopf differential.
\end{remark}

\subsection{Gauss map}\label{sec:gauss-maps}

Let $f^\Hh\colon\Sigma\to\Hh^3$ be a CMC $1$ conformal immersion and let $N$ be the normal of $f$. For each $z\in\Sigma$, the Riemannian exponential map at $f(z)$ in the direction of $N(z)$ defines a geodesic $\gamma\colon\Rr\to\Hh^3$. In Minkowski space $\Rr^{1,4}$, this geodesic is the intersection between the hyperboloid $\Hh^3$ and the linear plane spanned by $f(z)$ and $N(z)$. As $t$ tends to infinity, the asymptotic direction of $\gamma(t)$ can be identified with the null line through $f(z)+N(z)$. This null-line has a unit-length representative in the tangent space at $\I$:
\begin{equation}\label{eq:null-line}
	v_N := \frac{2(f+N)}{\tr(f+N)} - \I \in  T_{\I}\Hh^3.
\end{equation}
We identify $\Ss^2 \subset T_{\I}\Hh^3$ with the Riemann sphere via the stereographic projection $\St\colon\Ss^2\to\Cc\cup\{\infty\}$ defined by
\begin{equation}\label{eq:stereo}
	\St\inv(u/v) := \I-2\norm{\nu}^{-2}\nu \nu^*,\quad \nu:=(u,v).
\end{equation} 

\begin{definition}
	The \textbf{hyperbolic Gauss} map of a CMC 1 immersion is the map $G=\St\circ v_N$. The \textbf{Gauss map} of a minimal immersion $f^\Ee$ is simply $G = \St(N)$ where $N$ is the unit normal of $f^\Ee$.
\end{definition}

\begin{theorem}\label{theorem:gauss-map}
	Let $\Phi$ be an LWR frame with potential $\xi$, let $x$ be a spinor for $\xi$ and let $y:=\Phi x$. Let $f^\Ee$ and $f^\Hh$ be the immersions induced by $\Phi$ at the evaluation points $(\la_0,\la_1)$. The Gauss map of $f^\Ee$ and the hyperbolic Gauss map of $f^\Hh$ are both given by
	\begin{equation*}
		G = \frac{u}{v},\quad (u,v):=y_{\la_0}.
	\end{equation*} 
\end{theorem}
\begin{proof}	
	In $\Ee^3$, the formula is directly given by \eqref{eq:N-euclidean}.
	
	In $\Hh^3$, use the definition of $f$ in \eqref{eq:immersions} and recall \eqref{eq:orthonormal-frame-H3} to get
	\begin{equation*}
		f+N = \Psi^*(\I + n)\Psi.
	\end{equation*}
	Note that for all $x\in\Cc^2$,
	\begin{equation*}
		\widehat{xx^*} = \I - \norm{x}^{-2}xx^*.
	\end{equation*}
	Therefore, with \eqref{eq:n}:
	\begin{equation*}
		\I + n = 2\norm{y_1}^{-2}\widehat{y_1y_1^*}.
	\end{equation*}
	Noting that 
	\begin{equation*}
		\Psi^*\widehat{y_1y_1^*}\Psi = \widehat{y_0y_0^*}, \quad y_0:=y_{\la_0},
	\end{equation*}
	we get
	\begin{equation*}
		f+N = 2\norm{y_1}^{-2}\widehat{y_0y_0^*}.
	\end{equation*}
	Noting that $\tr(\widehat{xx^*})=1$ for all $x\in\Cc^2$,
	\begin{align*}
		v_N &= \frac{2(f+N)}{\tr(f+N)} - \I \\
		&= 2\widehat{y_0y_0^*} - \I \\
		&= 2\I - 2\norm{y_0}^{-2}y_0y_0^* - \I\\
		&=\St\inv(u/v).\qedhere
	\end{align*}
\end{proof}

\begin{corollary}\label{corollary:gaussmap}
	The Schwarzian derivative of $G$ satisfies
	\begin{equation*}
		\sS[G] = \la_0 q + s
	\end{equation*}
	where $q$ and $s$ are as in theorem~\ref{theorem:schwarz-potential}.
\end{corollary}
\begin{proof}
	This is a direct consequence of corollary~\ref{corrolary:schwarz-spinor}.
\end{proof}
 

\section{Related surfaces}\label{sec:related-surfaces}

Starting with an LWR potential $\xi$, in order to make an immersion via the LWR, one has to choose a pair of evaluation points $(\la_0, \la_1)$ and an initial condition $\Phi(z_0)$ for the LWR frame $\Phi$. This section investigates the consequences of these choices. We first focus on the evaluation points and show that moving them describes associated families. We then show how changing the initial condition for the LWR frame amounts to Goursat transformations of the induced surfaces. Noting that rigid motions induce a special class of Goursat transformations, we study them in this section.

\subsection{Moving the evaluation points}

We fix an LWR frame $\Phi$ and consider the evaluation points $(\la_0, \la_1)$ as free parameters. We want to study how the immersions $f^\Ee_{(\la_0, \la_1)}$ and $f^\Hh_{(\la_0, \la_1)}$ induced by LWR at $(\la_0, \la_1)$ change as $\la_0$ or $\la_1$ moves in $\Cc$.

\subsubsection{Associated family}

The classical associated family is parametrized by $\la\in\Ss^1\subset\Cc$. We extend this family to any $\la\in\Cc^*$.

\begin{definition}
	Let $f$ be a minimal or CMC $1$ conformal immersion with metric $ds^2$ and Hopf differential $Qdz^2$. The \textbf{complex associated family} $(f_\la)_{\la\in\Cc^*}$ of $f$ is the family of minimal or CMC $1$ immersions whose metric and Hopf differential read
	\begin{equation*}
		ds_\la^2 = \abs{\la}^2ds^2,\quad Q_\la dz^2 = \la Qdz^2
	\end{equation*}
	for all $\la\in \Cc^*$. By Bonnet's theorem, the complex associated family is defined up to ambient isometries.
\end{definition}

\begin{theorem}\label{theorem:complex-associated-family}
	The complex associated family of $f^\Ee$ is obtained by fixing $\la_0$ and moving $\la_1$. The complex associated family of $f^\Hh$ is obtained by fixing $\la_1$ and moving $\la_0$. 
\end{theorem}
\begin{proof}
	It is a direct consequence of theorem~\ref{theorem:fundamental-forms}.
\end{proof}

\subsubsection{Dual associated family}

\begin{definition}[\cite{umehara-yamada1997}]
	Let ${f} = {\Psi}^*{\Psi} $ be a conformal, CMC $1$ immersion into $\Hh^3$. The immersion $f^\sharp = (\Psi\inv)^*\Psi\inv$ is the \textbf{dual} of $f$. It is also a conformal, CMC $1$ immersion into $\Hh^3$.
\end{definition}

\begin{theorem}\label{theorem:dual}
	The dual of $f^\Hh$ is obtained by permuting $\la_0$ and $\la_1$.
\end{theorem}
\begin{proof}
	By \eqref{eq:null-curves}, permuting the evaluation points amounts to taking the inverse null curve, which is what duality does.
\end{proof}

\begin{definition}\label{def:dual-assoc-family}
	The \textbf{dual associated family} of a conformal, CMC $1$ immersion $f$ into $\Hh^3$ is the family $(\hat{f}_\la)_{\la\in\Cc^*}$ defined by
	\begin{equation*}
		\hat{f}_\la := \left((f^\sharp)_\la\right)^\sharp
	\end{equation*}
	where $(\cdot)^\sharp$ denotes the dual and $(\cdot)_\la$ denotes the complex associated family.
\end{definition}

\begin{corollary}
	The dual associated family of $f^\Hh$ is obtained by fixing $\la_0$ and moving $\la_1$. 
\end{corollary}
\begin{proof}
	Apply successively theorem~\ref{theorem:dual}, theorem~\ref{theorem:complex-associated-family} and theorem~\ref{theorem:dual}.
\end{proof}

\begin{figure}[h]
	\centering
	\begin{subfigure}{0.19\textwidth}
		\includegraphics[width=\linewidth]{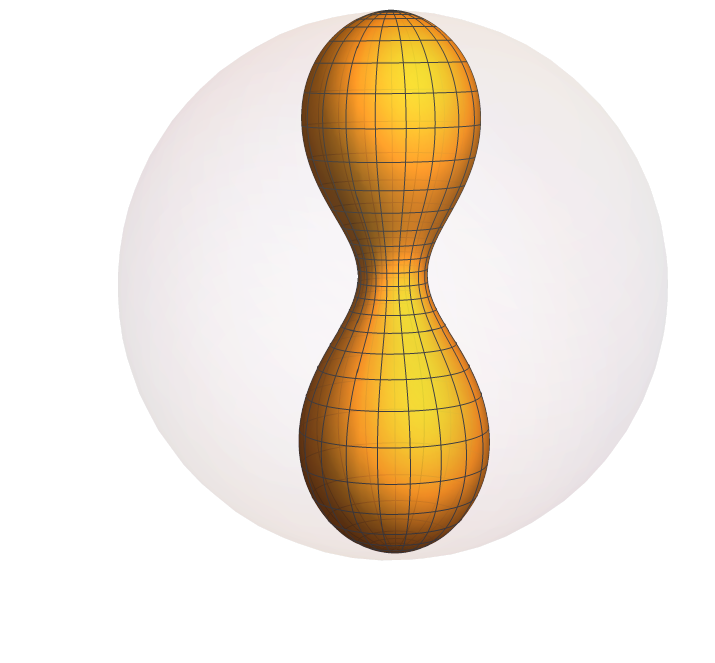}
		\caption{$\la_1 = 1$}
		\label{fig:cat1}
	\end{subfigure}
	\hfill
	\begin{subfigure}{0.19\textwidth}
		\includegraphics[width=\linewidth]{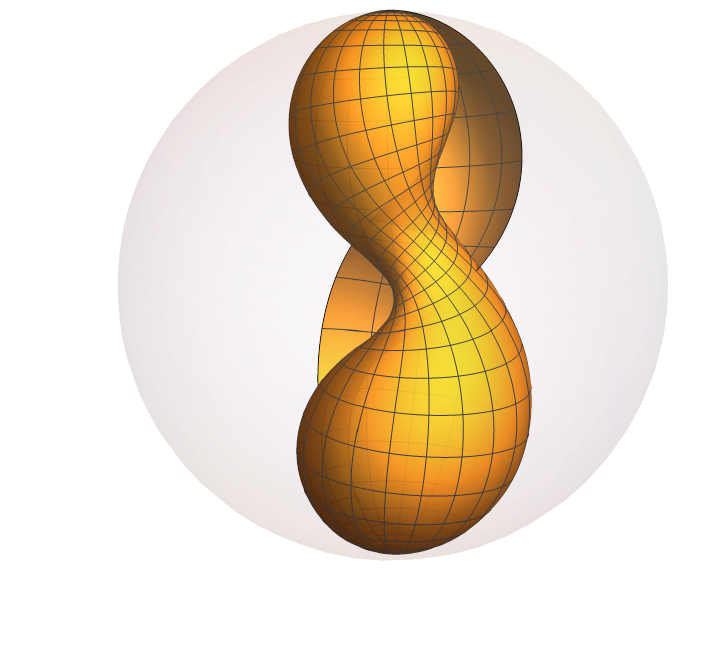}
		\caption{$\la_1 = e^{\frac{i\pi}{4}}$}
		\label{fig:cat2}
	\end{subfigure}
	\hfill
	\begin{subfigure}{0.19\textwidth}
		\includegraphics[width=\linewidth]{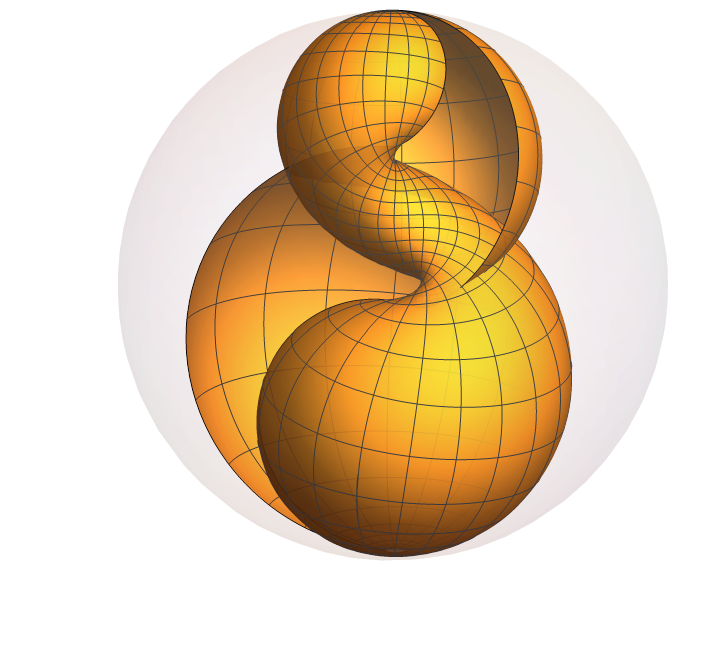}
		\caption{\centering$\la_1 = i$}
		\label{fig:cat3}
	\end{subfigure}
	\hfill
	\begin{subfigure}{0.19\textwidth}
		\includegraphics[width=\linewidth]{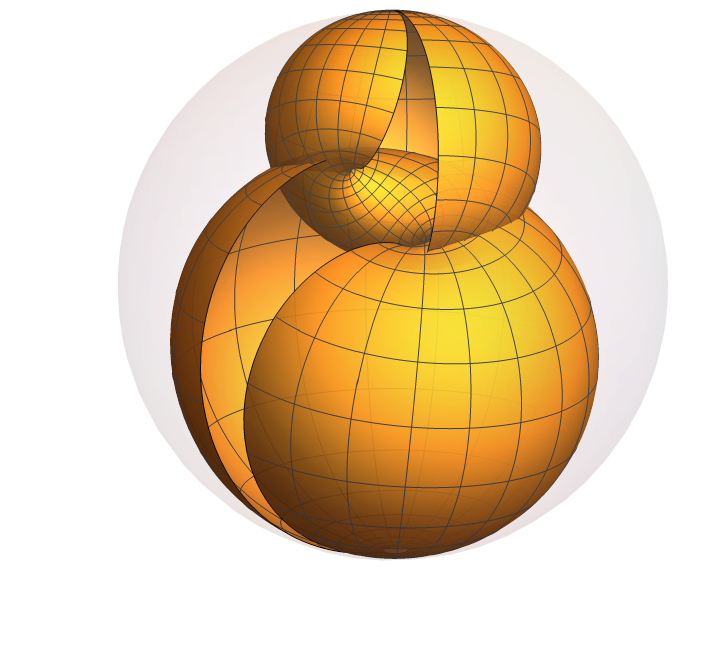}
		\caption{$\la_1 = e^{\frac{3i\pi}{4}}$}
		\label{fig:cat4}
	\end{subfigure}
	\hfill
	\begin{subfigure}{0.19\textwidth}
		\includegraphics[width=\linewidth]{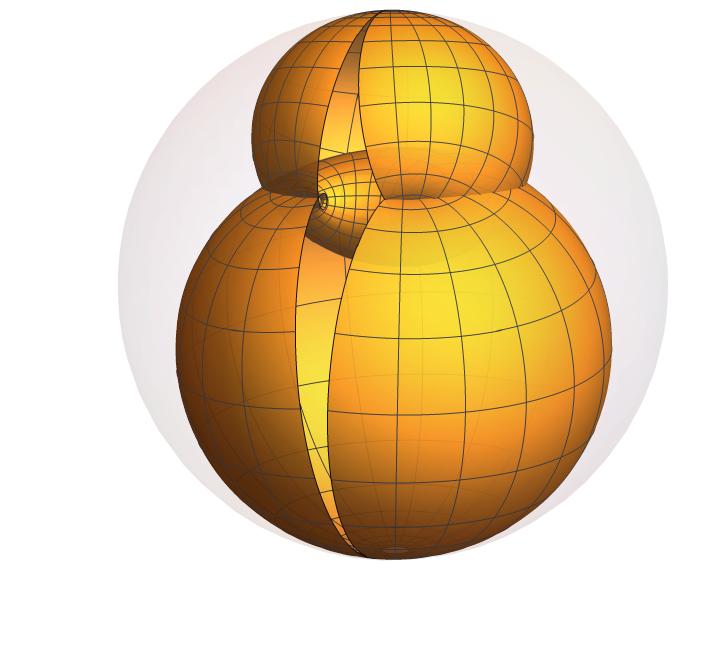}
		\caption{$\la_1 = -1$}
		\label{fig:cat5}
	\end{subfigure}
	\caption{Five members in the dual associated family of a catenoid (definition~\ref{def:dual-assoc-family}).\\
	The evaluation point $\la_0=0$ is fixed while $\la_1$ moves around the origin. All these immersions have the same hyperbolic Gauss map.}
	\label{fig:dual-assoc-cat}
\end{figure}

The following proposition shows that a minimal immersion can be obtained as the blow-up limit in the dual associated family of a CMC $1$ immersion. 

\begin{proposition}[Blow-up limit of CMC $1$ surfaces]\label{proposition:blow-up}
	Let $\gamma\colon(-\epsilon,\epsilon)\to \Cc$ such that $\gamma(0) = \la_0$ and ${\gamma}'(0) = \la_1-\la_0\neq 0$. Let $\Phi$ be an LWR frame and let $\Psi_t := \Phi_{\gamma(t)}\Phi_{\gamma(0)}\inv$. Define $f^\Hh_t := \Psi_t^*\Psi_t$. Then
	\begin{equation*}
		\frac{d}{dt}\left(f^\Hh_t\right)_{\mid t=0} = f^\Ee,
	\end{equation*}
	where $f^\Ee$ is the minimal immersion induced by the LWR frame at the evaluation points $(\la_0, \la_1)$.
\end{proposition}
\begin{proof}
	Compute
	\begin{equation*}
		\frac{d}{dt}\left(\Phi_{\gamma(t)}\right)_{\mid t=0} = {\gamma}'(0)\dot{\Phi}_{\gamma(0)} = (\la_1-\la_0)\dot{\Phi}_{\la_0}.
	\end{equation*}
	Therefore,
	\begin{equation*}
		\frac{d}{dt}(\Psi_t)_{\mid t=0} = (\la_1-\la_0)\dot{\Phi}_{\la_0}\Phi_{\la_0}\inv = \psi
	\end{equation*}
	where $\psi$ is the holomorphic null curve induced by the LWR frame $\Phi$ at the evaluation points $(\la_0, \la_1)$.
	Furthermore,
	\begin{equation*}
		\frac{d}{dt}(f_t^\Hh)_{\mid{t=0}} = \frac{d}{dt}(\Psi_t)_{\mid t=0}^* +  \frac{d}{dt}(\Psi_t)_{\mid t=0} = \psi^* + \psi = f^\Ee
	\end{equation*}
	because $\Psi_0 = \I$.
\end{proof}

\subsection{Holomorphic dressing}

\begin{definition}
	Let $\Phi$ be an LWR frame and let $g_\la\in\SL_2\Cc$ be constant in $z$ (but depending on $\la$). 
	\begin{itemize}
		\item \textbf{Dressing} $\Phi$ by $g$ is finding a map $h$ (depending on both $z$ and $\la$) such that $\hat{\Phi}:=g\Phi h\inv$ is an LWR frame.
		\item If $g$ is holomorphic in $\la\in\Cc$, a dressing is given by choosing $h=\I$. In this case, the map $\Phi\mapsto\hat{\Phi}$ is a \textbf{holomorphic dressing}. Note that holomorphic dressing is a group action of $\SL_2\Cc$-valued $\la$-holomorphic maps on LWR frames.
	\end{itemize}
\end{definition}

In this section, we study holomorphic dressing and show that they induce Goursat transformations (including rigid motions). We will consider simple factor dressing in section~\ref{sec:simple-factor-dressing}.

\begin{definition}[\cite{Hertrich-Jeromin_2003}, Lemma 5.3.1]
	Two minimal or CMC $1$ immersions are related by a \textbf{Goursat transformation} if they have the same Hopf differential and if the Schwarzian derivative of their Gauss maps are equal.
\end{definition}

\begin{theorem}\label{theorem:goursat}
	Holomorphic dressing induces a Goursat transformation of the induced immersion $f$. Moreover, any Goursat transformation of $f$ can be obtained by holomorphic dressing, provided that $f$ is not flat.
\end{theorem}
\begin{proof}
	Holomorphic dressing amounts to changing the initial condition $\Phi(z_0)$ in the initial value problem~\eqref{eq:cauchy-problem}. Therefore, it does not change the potential $\xi$. By remark~\ref{remark:hopf}, the Hopf differential is unchanged.
	Moreover, by theorem~\ref{theorem:fundamental-forms}, the Gauss map $\widehat{G}$ of the dressed immersion $\widehat{f}$ reads
	\begin{equation*}
		\widehat{G} = \frac{a G + b}{c G + d},\quad R_{\la_0} = \matrix{a}{b}{c}{d}\in\SL_2\Cc
	\end{equation*}
	where $G$ is the Gauss map of $f$. Therefore, $\widehat{G}$ and $G$ differ by a M\"obius transformation, so they have the same Schwarzian derivative.
	
	Conversely, let $\widehat{f}$ be a Goursat transformation of $f$. 
	The immersions are not flat, so their respective potential can be gauged into a Schwarz potential (by theorem~\ref{theorem:schwarz-potential}). By corollary~\ref{corrolary:schwarz-spinor}, these Schwarz potentials agree because the respective Hopf differential and Schwarzian derivative of the Gauss map agree for a Goursat pair. Therefore, the gauged LWR frames agree, up to a choice of initial condition, and one can be obtained from the other via a holomorphic dressing.
\end{proof}

\subsection{Rigid motions}

\begin{definition}
	Let $f\colon\Sigma\to\Xx^3$ be an immersion. A \textbf{rigid motion} of $f$ is an immersion $\widehat{f}\colon\Sigma\to\Xx^3$ such that there exists an orientation-preserving isometry $\jJ$ of $\Xx^3$ satisfying $\widehat{f} = \jJ\circ f$. 
\end{definition}

\begin{theorem}\label{theorem:rigid-motions}
	Let $f$ be the immersion induced by the LWR frame $\Phi$ at $(\la_0, \la_1)$ and let $\hat{\Phi} = R\Phi$ be a holomorphic dressing of $\Phi$.
	\begin{itemize}
		\item In $\Ee^3$, if $R_{\la_0}\in\SU_2$, then the immersion induced by $\hat\Phi$ is a rigid motion of $f$.
		\item In $\Hh^3$, if $R_{\la_1}\in\SU_2$, then the immersion induced by $\hat\Phi$ is a rigid motion of $f$.
		\item Any rigid motion of $f$ can be obtained this way, provided that $f$ is not flat.
	\end{itemize}
\end{theorem}
\begin{proof}
	The null curves $\hat{\psi}$ and $\hat{\Psi}$ induced by $\hat{\Phi}$ read
	\begin{equation*}
		\hat{\psi} = R_{\la_0}\psi R_{\la_0}\inv + (\la_1-\la_0)(\dot{R}R\inv)_{\la_0}
	\end{equation*}
	and
	\begin{equation*}
		\hat{\Psi} = R_{\la_1}\Psi R_{\la_0}\inv.
	\end{equation*}
	Therefore, if $R_{\la_0}\in\SU_2$, then the induced immersion in $\Ee^3$ reads
	\begin{equation}\label{eq:rigid-motion-E3}
		\hat{f}^\Ee = R_{\la_0} f^\Ee R_{\la_0}\inv + T
	\end{equation}
	where
	\begin{equation}\label{eq:rigid-motion-E3-T}
		T = c+c^*,\quad c = (\la_1-\la_0)(\dot{R}R\inv)_{\la_0}.
	\end{equation}
	If $R_{\la_1}\in\SU_2$, then
	\begin{equation}\label{eq:rigid-motion-H3}
		\hat{f}^\Hh = (R_{\la_0}\inv)^* f^\Hh R_{\la_0}\inv.
	\end{equation}
	Therefore, in both spaces, $\hat{f}$ is a rigid motion of $f$.
	
	Conversely, let $\hat{f}$ be a rigid motion of $f$. Rigid motions are a special case of Goursat transformations, and $f$ is not flat. Therefore, by theorem~\ref{theorem:goursat}, $\hat{f}$ can be obtained by dressing. Let $\hat{\Phi} = R\Phi$ inducing $\hat{f}$. Because $\hat{f}$ is a rigid motion of $f$, the corresponding metrics are the same. Therefore, in $\Ee^3$, by theorem~\ref{theorem:fundamental-forms}, with $y=\Phi x$, 
	\begin{equation*}
		\norm{y} = \norm{R_{\la_0}y}
	\end{equation*}
	constantly on $\Sigma$.
	Using that $f$ is not flat, we deduce that $R_{\la_0}\in\SU_2$ in $\Ee^3$.
	By the same argument, $R_{\la_1}\in\SU_2$ in $\Hh^3$.
\end{proof}

\begin{corollary}[Identity]\label{cor:identical-immersions}
	Let $f$ and $\hat{f}$ be two non-flat immersions obtained via the LWR at the same evaluation points $(\la_0, \la_1)$. Then $f=\hat{f}$ if and only if the corresponding LWR frames satisfy
	\begin{equation*}
		{\Phi} = R\hat{\Phi} g
	\end{equation*}
	for some LWR gauge $g$ and some holomorphic dressing $R$ such that:
	\begin{itemize}
		\item in $\Ee^3$, $R_{\la_0}=\pm \I$ and $(\la_1-\la_0)\dot{R}_{\la_0}\in\su_2$,
		\item in $\Hh^3$, $R_{\la_0}=\pm \I$ and ${R}_{\la_1}\in\SU_2$.
	\end{itemize}
\end{corollary}
\begin{proof}
	If $\hat{f}={f}$, then the two immersions have the same Hopf differential and the Schwarzian derivative of their Gauss map are the same. Therefore, the Schwarz form of their potential is the same. Thus, their respective potential belong to the same gauge class: there exists a gauge $g$ such that
	\begin{equation*}
		\hat{\xi} = \xi\cdot g.
	\end{equation*}
	Solving the corresponding differential system yields
	\begin{equation*}
		\hat{\Phi} = R\Phi g
	\end{equation*}
	for some $\la$-holomorphic $R\colon\Cc\to\SL_2\Cc$. By the same argument as in the proof of theorem~\ref{theorem:rigid-motions}, in $\Ee^3$, $R_{\la_0}\in\SU_2$ and $\hat{f}$ is given by \eqref{eq:rigid-motion-E3}. But $\hat{f}=f$ and the immersion is not flat, so $R_{\la_0}=\pm\I$ and $c+c^*=0$ where $c := (\la_1-\la_0)(\dot{R}R\inv)_{\la_0}$, implying that $(\la_1-\la_0)\dot{R}_{\la_0}\in\su_2$. The same arguments can be applied in the case of $\Hh^3$.
	
	The converse is a consequence of lemma~\ref{lemma:gauging-lwr}, theorem~\ref{theorem:rigid-motions}, equations~\eqref{eq:rigid-motion-E3}--\eqref{eq:rigid-motion-E3-T} and \eqref{eq:rigid-motion-H3}.
\end{proof}

We end this section by exhibiting a group homomorphism between holomorphic dressings and rigid motions. This homomorphism will be useful in section~\ref{sec:catenoids} in order to distinguish rotations from general screw motions.

\begin{proposition}\label{prop:rigid-motion-homomorphism-E3}
	Fix $(\la_0,\la_1)$. Let $G_{\la_0}$ be the subgroup of $\la$-holomorphic maps $R\colon\Cc\to\SL_2\Cc$ such that $R_{\la_0}\in\SU_2$. Consider the map $\phi\colon G_{\la_0}\to \Iso^+(\Ee^3)$ where for all $X\in\Ee^3$,
	\begin{equation*}
		\phi(R)(X) := R_{\la_0}XR_{\la_0}\inv + \rho(R) + \rho(R)^*
	\end{equation*}
	with
	\begin{equation*}
		\rho(X) := (\la_1-\la_0)(\dot{R}R\inv)_{\la_0}.
	\end{equation*}
	The map $\phi$ is a group homomorphism.
\end{proposition}
\begin{proof}
	Let $R,S\in G_{\la_0}$ and $X\in\Ee^3$.
	\begin{equation*}
		\phi(RS)(X) = (RS)_{\la_0}X(RS)_{\la_0}\inv + \rho(RS) + \rho(RS)^*.
	\end{equation*}
	But
	\begin{equation*}
		\rho(RS) = (\la_1-\la_0)(\dot{R}S + R\dot{S})_{\la_0}(RS)_{\la_0}\inv =\rho(R) + R_{\la_0}\rho(S)R_{\la_0}\inv.
	\end{equation*}
	With $R_{\la_0}\in\SU_2$,
	\begin{equation*}
		\rho(RS) + \rho(RS)^* = \rho(R) + \rho(R)^* + R_{\la_0}(\rho(S) + \rho(S)^*)R_{\la_0}\inv.
	\end{equation*}
	Therefore, $\phi(RS) = \phi(R)\phi(S)$.
\end{proof}
\begin{proposition}\label{prop:rigid-motion-homomorphism-H3}
	Fix $(\la_0,\la_1)$. Let $G_{\la_1}$ be the subgroup of $\la$-holomorphic maps $R\colon\Cc\to\SL_2\Cc$ such that $R_{\la_1}\in\SU_2$. Consider the map $\phi\colon G_{\la_1}\to \Iso^+(\Hh^3)$ where for all $X\in\Hh^3$,
	\begin{equation*}
		\phi(R)(X) := (R_{\la_0}\inv)^*XR_{\la_0}\inv.
	\end{equation*}
	The map $\phi$ is a group homomorphism.
\end{proposition}
\begin{proof}
	The proof amounts to checking that
	\begin{equation*}
		((RS)\inv)^* = (R\inv)^*(S\inv)^*
	\end{equation*}
	for all $R,S\in\SL_2\Cc$.
\end{proof}


\section{Symmetries and closing}\label{sec:symmetries}

In this section, we fix the LWR data $(\Sigma, \xi, \Phi, \la_0, \la_1)$ and look for conditions on the frame $\Phi$ such that the induced immersion $f$ admits a symmetry.

\begin{definition}
	Let $f\colon\Sigma\to\Xx^3$ be a conformal immersion. A \textbf{symmetry} of $f$ is a pair $(\tau, \jJ)$ where $\tau$ is a diffeomorphism of $\Sigma$ and $\jJ$ is an ambient isometry such that
	\begin{equation*}
		\tau^* f = \jJ \circ f.
	\end{equation*}
\end{definition}

In Euclidean space, four cases can occur depending on the orientations of $\tau$ and $\jJ$. They are described in section~\ref{sec:symmetries-E3} and summarized in Table~\ref{table:1}. In hyperbolic space (section~\ref{sec:symmetries-H3}), only the diagonal entries of Table~\ref{table:1} occur (see remark~\ref{remark:orientation-reversing-H3}). We give closing conditions for $f$ in terms of the monodromy of $\Phi$ in Section~\ref{sec:closing}.

\begin{table}[h!]
	\centering
	\renewcommand{\arraystretch}{1.5}
	\begin{tabular}{ c|c|c } 
		& $\tau$ holom. & $\tau$ anti-holom. \\ 
		\hline
		$\jJ$ orient. preserving & $R\Phi g$ & $\bar{(R\Phi g)_{-\bar{\la}}}$ \\ 
		\hline
		$\jJ$ orient. reversing & $(R\Phi g)_{-\la}$ & $\bar{(R\Phi g)_{\bar{\la}}}$ \\ 
	\end{tabular}
	\caption{$\tau^*\Phi$, when $\tau^*f = \jJ\circ f$.}
	\label{table:1}
\end{table}

\subsection{Symmetries in $\Hh^3$}\label{sec:symmetries-H3}

\begin{theorem}\label{theorem:symmetries-H3}
	Assume that $f^\Hh$ is not flat. The following statements are equivalent.
	\begin{enumerate}
		\item There exist a biholomorphism $\tau$ of $\Sigma$ and an orientation-preserving isometry $\jJ$ of $\Hh^3$ such that 
		\begin{equation*}
			\tau^* f^\Hh = \jJ\circ f^\Hh.
		\end{equation*}
		\item There exist an LWR gauge $g$ and a holomorphic dressing $R$ such that $R_{\la_1}\in\SU_2$ and
		\begin{equation*}
			\tau^*\Phi = R\Phi g.
		\end{equation*}
	\end{enumerate}
\end{theorem}
\begin{proof}
	Assume that $(1)$ holds. Then $\tau^*f^\Hh$ is a rigid motion of $f^\Hh$, so by theorem~\ref{theorem:rigid-motions}, there exists a holomorphic dressing $\widetilde{R}$ such that $\widetilde{R}_{\la_1}\in\SU_2$ and 
	\begin{equation*}
		\tau^* f^\Hh = \LWR(\Sigma, \xi, \widetilde{R}\Phi, \la_0, \la_1).
	\end{equation*}
	On the other hand,
	\begin{equation*}
		\tau^* f^\Hh = \LWR(\Sigma, \tau^*\xi, \tau^*\Phi, \la_0, \la_1).
	\end{equation*}
	By corollary~\ref{cor:identical-immersions}, there exists an LWR gauge $g$ and a holomorphic dressing $\widehat{R}$ such that $\widehat{R}_{\la_1}\in\SU_2$ and 
	\begin{equation*}
		\tau^*\Phi = \widehat{R}\widetilde{R}\Phi g.
	\end{equation*}
	Conclude with $R := \widehat{R}\widetilde{R}$.
	
	Conversely, assume that $(2)$ holds. Then the LWR potential associated with $\tau^*\Phi$ is $\xi\cdot g$. Therefore,
	\begin{align*}
		\tau^*f^\Hh &= \LWR(\Sigma, \tau^*\xi, \tau^*\Phi, \la_0, \la_1)\\
		&=  \LWR(\Sigma, \xi\cdot g, R\Phi g, \la_0, \la_1)\\
		&= \LWR(\Sigma, \xi, R\Phi, \la_0, \la_1),
	\end{align*}
	and by theorem~\ref{theorem:rigid-motions}, there exists an orientation-preserving isometry $\jJ$ of $\Hh^3$ such that $\tau^*f^\Hh = \jJ\circ f^\Hh$.
\end{proof}

A similar statement holds with an anti-holomorphic $\tau$:
\begin{theorem}\label{theorem:symmetries-H3-reversing}
	Assume that $f^\Hh$ is a not flat and obtained via the LWR at $(\la_0,\la_1)\in\Rr^2$.
	The following statements are equivalent.
	\begin{enumerate}
		\item There exists an anti-holomorphic diffeomorphism $\tau$ of $\Sigma$ and an orientation-reversing isometry $\jJ$ of $\Hh^3$ such that 
		\begin{equation*}
			\tau^* f^\Hh = \jJ\circ f^\Hh.
		\end{equation*}
		\item There exist an LWR gauge $g$ and a holomorphic dressing $R$ such that $R_{\la_1}\in\SU_2$ and
		\begin{equation*}
			\bar{\tau^*\Phi_{\bar{\la}}} = R_\la\Phi_\la g.
		\end{equation*}
	\end{enumerate}
\end{theorem}
\begin{proof}
	Assume that $(1)$ holds. Then there exists an orientation-preserving isometry $\lL$ such that
	\begin{equation*}
		\bar{\tau^*f^\Hh} = \lL\circ f^\Hh.
	\end{equation*}
	By theorem~\ref{theorem:rigid-motions}, there exists a holomorphic dressing $\widetilde{R}$ such that $\widetilde{R}_{\la_1}\in\SU_2$ and 
	\begin{equation*}
		\bar{\tau^*f^\Hh} = \LWR(\Sigma, \xi, \widetilde{R}\Phi, \la_0, \la_1).
	\end{equation*}
	On the other hand, defining $\widehat{\xi}_\la = \bar{\tau^*\xi_\bar{\la}}$ and $\widehat{\Phi}_\la = \bar{\tau^*\Phi_\bar{\la}}$, we have that $\widehat{\xi}$ and $\widehat{\Phi}$ are respectively LWR potential and frame, and going through the steps of LWR gives
	\begin{equation*}
		\widehat{\Psi} = \widehat{\Phi}_{\la_1}\widehat{\Phi}_{\la_0}\inv = \tau^*\bar{\Psi}
	\end{equation*}
	because $\la_0, \la_1\in\Rr$. Moreover,
	\begin{equation*}
		\widehat{f} = \widehat{\Psi}\widehat{\Psi}^* = \bar{\tau^*f^\Hh}.
	\end{equation*}
	Therefore,
	\begin{equation*}
		\bar{\tau^*f^\Hh} = \LWR(\Sigma, \widehat{\xi}, \widehat{\Phi}, \la_0, \la_1).
	\end{equation*} 
	We deduce that $\bar{\tau^*f^\Hh}$ is induced by both $\widetilde{R}\Phi$ and $\bar{\tau^*\Phi_\bar{\la}}$. By corollary~\ref{cor:identical-immersions}, there exist an LWR gauge $g$ and a holomorphic dressing $\widehat{R}$ such that $\widehat{R}_{\la_1}\in\SU_2$ and 
	\begin{equation*}
		\bar{\tau^*\Phi_\bar{\la}} = \widehat{R}\widetilde{R} \Phi g.
	\end{equation*}
	Conclude with $R:=\widehat{R}\widetilde{R}$.
	
	The converse is proved with the same arguments as in the proof of theorem~\ref{theorem:symmetries-H3}.
\end{proof}

\subsection{Symmetries in $\Ee^3$}\label{sec:symmetries-E3}

Each of theorems~\ref{theorem:symmetries-E3} to~\ref{theorem:symmetries-E3-minus-bar} correspond to an entry of Table~\ref{table:1}.

\begin{theorem}\label{theorem:symmetries-E3}
	Assume that $f^\Ee$ is not flat.  The following statements are equivalent.
	\begin{enumerate}
		\item There exist a biholomorphism $\tau$ of $\Sigma$ and an orientation-preserving isometry $\jJ$ of $\Ee^3$ such that 
		\begin{equation*}
			\tau^* f^\Ee = \jJ\circ f^\Ee.
		\end{equation*}
		\item There exist an LWR gauge $g$ and a holomorphic dressing $R$ such that $R_{\la_0}\in\SU_2$ and
		\begin{equation*}
			\tau^*\Phi = R\Phi g.
		\end{equation*}
	\end{enumerate}
\end{theorem}
\begin{proof}
	Follow step by step the proof of theorem~\ref{theorem:symmetries-H3} to prove this theorem.
\end{proof}

\begin{theorem}\label{theorem:symmetries-E3-bar}
	Assume that $f^\Ee$ is not flat and obtained via the LWR at $(\la_0,\la_1)\in\Rr^2$.
	The following statements are equivalent.
	\begin{enumerate}
		\item There exist an anti-holomorphic diffeomorphism $\tau$ of $\Sigma$ and an orientation-reversing isometry $\jJ$ of $\Ee^3$ such that 
		\begin{equation*}
			\tau^* f^\Ee = \jJ\circ f^\Ee.
		\end{equation*}
		\item There exist an LWR gauge $g$ and a holomorphic dressing $R$ such that $R_{\la_0}\in\SU_2$ and
		\begin{equation*}
			\bar{\tau^*\Phi_{\bar{\la}}} = R_\la\Phi_\la g.
		\end{equation*}
	\end{enumerate}
\end{theorem}
\begin{proof}
	Follow step by step the proof of theorem~\ref{theorem:symmetries-H3-reversing} to prove this theorem.
\end{proof}

We get two more types of symmetries in the Euclidean case.

\begin{theorem}\label{theorem:symmetries-E3-minus}
	Assume that $f^\Ee$ is not flat and obtained via the LWR with $\la_0=0$.
	The following statements are equivalent.
	\begin{enumerate}
		\item There exist a biholomorphism $\tau$ of $\Sigma$ and an orientation-reversing isometry $\jJ$ of $\Ee^3$ such that 
		\begin{equation*}
			\tau^* f^\Ee = \jJ\circ f^\Ee.
		\end{equation*}
		\item There exist an LWR gauge $g$ and a holomorphic dressing $R$ such that $R_{0}\in\SU_2$ and
		\begin{equation*}
			\tau^*\Phi_{-\la} = R_\la\Phi_\la g.
		\end{equation*}
	\end{enumerate}
\end{theorem}
\begin{proof}
	Again, the ideas are the same as in theorem~\ref{theorem:symmetries-H3-reversing}, noting that the isometry $X\mapsto -X$ in $\Ee^3$ is orientation reversing, and that with $\widehat{\Phi}_\la = \Phi_{-\la}$ and $\la_0 = 0$, the induced null curve reads
	\begin{align*}
		\widehat{\psi} &= (\la_1-\la_0)\dot{\widehat{\Phi}}_{\la_0}\widehat{\Phi}_{\la_0}\inv \\
		&= (\la_0-\la_1) \dot{\Phi}_{\la_0}\Phi_{\la_0} \inv\\
		&=-\psi.\qedhere
	\end{align*}
\end{proof}

\begin{theorem}\label{theorem:symmetries-E3-minus-bar}
	Assume that $f^\Ee$ is not flat and obtained via the LWR at $(\la_0, \la_1)$ with $\la_0=0$ and $\la_1\in\Rr$.
	The following statements are equivalent.
	\begin{enumerate}
		\item There exist an anti-holomorphic diffomorphism $\tau$ of $\Sigma$ and an orientation preserving isometry $\jJ$ of $\Ee^3$ such that 
		\begin{equation*}
			\tau^*f^\Ee = \jJ\circ f^\Ee.
		\end{equation*}
		\item There exist an LWR gauge $g$ and a holomorphic dressing $R$ such that $R_{0}\in\SU_2$ and
		\begin{equation*}
			\bar{\tau^*\Phi_{-\bar{\la}}} = R_\la\Phi_\la g.
		\end{equation*}
	\end{enumerate}
\end{theorem}
\begin{proof}
	Combine theorems~\ref{theorem:symmetries-E3-bar} and~\ref{theorem:symmetries-E3-minus}.
\end{proof}

\begin{remark}\label{remark:orientation-reversing-H3}
	Theorems~\ref{theorem:symmetries-E3-minus} and~\ref{theorem:symmetries-E3-minus-bar} do not have an equivalent in $\Hh^3$ because if $\tau$ is anti-holomorphic and $\jJ$ is orientation preserving (or if $\tau$ is holomorphic and $\jJ$ is orientation reversing), then the mean curvature of $\tau^*f$ and the mean curvature of $\jJ\circ f$ have opposite sign. Therefore, only one of them can be obtained via the LWR.
\end{remark}

\subsection{Closing conditions}\label{sec:closing}

Let $\pi_1(\Sigma)$ denote the first fundamental group of $\Sigma$ based at some point $z_0\in\Sigma$.

\begin{definition}
	For any loop $\gamma\in\pi_1(\Sigma)$, there exists a holomorphic dressing $M(\gamma)$ such that
	\begin{equation*}
		\tau^*\Phi =M(\gamma)  \Phi
	\end{equation*}
	where $\tau$ is the deck transformation associated with $\gamma$.
	The $\la$-holomorphic matrix $M(\gamma)$ is the \textbf{monodromy} of $\Phi$ along $\gamma$.
\end{definition}

\begin{theorem}\label{theorem:closing-conditions}
	Let $M$ be the monodromy of $\Phi$.
	\begin{itemize}
		\item $f^\Ee$ descends to a well-defined immersion of $\Sigma$ if and only if for all $\gamma\in\pi_1(\Sigma)$, 
		\begin{equation*}
			M_{\la_0}(\gamma)\in\{\pm \I\}\quad \text{and}\quad (\la_1-\la_0)\dot{M}_{\la_0}(\gamma)\in\su_2.
		\end{equation*}
		\item $f^\Hh$ descends to a well-defined immersion of $\Sigma$ if and only if for all $\gamma\in\pi_1(\Sigma)$, 
		\begin{equation*}
			M_{\la_0}(\gamma)\in\{\pm \I\}\quad \text{and}\quad M_{\la_1}(\gamma)\in\SU_2.
		\end{equation*}
	\end{itemize}
\end{theorem}
\begin{proof}
	It is a direct consequence of theorems~\ref{theorem:symmetries-E3} and~\ref{theorem:symmetries-H3} applied to every deck transformation $\tau$ of $\widetilde{\Sigma}$, and of corollary~\ref{cor:identical-immersions}.
\end{proof}


\section{Intrinsic surfaces of revolution}\label{sec:examples}

In this section, $\widetilde{\Cc^*}$ is the universal cover of $\Cc^*$.

\begin{definition}\label{definition:intrinsic-surf-rev}
	An immersion $f\colon \widetilde{\Cc^*}\to\Xx^3$ is an \textbf{intrinsic surface of revolution} if there exists a function $m\colon \widetilde{\Cc^*}\to\Rr_{>0}$, depending only on $\abs{z}$, such that the metric of $f$ reads
	\begin{equation*}
		ds^2 = m^2\abs{dz}^2.
	\end{equation*} 
\end{definition}

We construct all minimal and CMC $1$ intrinsic surfaces of revolution with the LWR. We first solve the Gauss-Codazzi equations for these surfaces. proposition~\ref{proposition:metric-hopf-revolution} together with Bonnet's theorem implies the existence of various minimal and CMC $1$ surfaces of revolution. We then show how to explicitly construct them via LWR (section~\ref{sec:surf-rev-generic}). Special cases occur when they descend to $\Cc^*$ and conformally extend to $z=0$ (Enneper surfaces, see section~\ref{sec:enneper}). Another special case are extrinsic surfaces of revolution (catenoids, see section~\ref{sec:catenoids}).

\subsection{Solution to the Gauss-Codazzi equations}

\begin{proposition}\label{proposition:metric-hopf-revolution}
	Let $f$ be a minimal or CMC $1$ intrinsic surface of revolution.
	\begin{itemize}
		\item If $f$ is flat, then there exist $a>0$ and $\alpha\in\Rr$ such that the metric of $f$ is
		\begin{equation}\label{eq:metric-revolution-flat}
			ds^2 = a^4\abs{z}^{4\alpha}\abs{dz}^2.
		\end{equation}
		\item If $f$ is not flat, then there exist $a,b>0$, $\alpha<\beta\in\Rr$ and $\nu\in[0,2\pi)$ such that the metric and Hopf differential of $f$ are
		\begin{equation}\label{eq:metric-revolution}
			ds^2 = \left(a^2 \abs{z}^{2\alpha} + b^2 \abs{z}^{2\beta}\right)^2\abs{dz}^2,
		\end{equation}
		\begin{equation}\label{eq:hopf-revolution}
			Qdz^2 = e^{i\nu}ab(\beta - \alpha)z^{\alpha+\beta -1}dz^2.
		\end{equation}
	\end{itemize}
\end{proposition}
\begin{proof}
	Write $z = re^{i\theta}$ and $ds^2 = 4e^{2\omega}\abs{dz}^2$ with $\omega\colon\widetilde{\Cc^*}\to\Rr$. 
	By assumption on $f$, the function $\omega$ only depends on $r$.
	The Gauss-Codazzi equations \eqref{eq:gauss-codazzi} read
	\begin{equation}\label{eq:gauss-codazzi-revolution}
		\omega_{rr} + r\inv \omega_r = \abs{Q}^2e^{-2\omega},\quad Q_{\zbar}=0.
	\end{equation}

		If $f$ is flat, $Q$ is constantly vanishing.
		In this case, all the solutions to \eqref{eq:gauss-codazzi-revolution} are given by
		\begin{equation*}
			\omega(r) = c_1 + c_2\log(r),\quad c_1,c_2\in\Rr. 
		\end{equation*}
		With $a:=2e^{c_1}>0$ and $\alpha := c_2\in\Rr$, the metric reads as in \eqref{eq:metric-revolution-flat}.
		
		If $f$ is not flat, then $Q$ is not constantly vanishing.
		The Codazzi equation (\eqref{eq:gauss-codazzi-revolution}) implies that $Q$ is holomorphic.
		Therefore, there exists a domain $D\subset\widetilde{\Cc^*}$ on which $Q$ never vanishes.
		On $D$, consider $h=\abs{Q}^2$ and compute
		\begin{equation}\label{eq:hinvh_theta}
			h\inv h_\theta = ih\inv (zh_z - \bar{z}h_\zbar) = i\left(zQ\inv Q_z - \bar{zQ\inv Q_z}\right).
		\end{equation}
		The Gauss equation in \eqref{eq:gauss-codazzi-revolution} implies that $\abs{Q}^2$ only depends on $r$, so $h_\theta = 0$ and by \eqref{eq:hinvh_theta},
		\begin{equation*}
			zQ\inv Q_z = \bar{z Q\inv Q_z}.
		\end{equation*}
		The left-hand side is holomorphic whereas the right-hand side is anti-holomorphic, so there exists $\gamma\in \Rr$ such that
		\begin{equation*}
			zQ\inv Q_z = \gamma.
		\end{equation*}
		Solving this equation for $Q$ yields for all $z\in D$,
		\begin{equation*}
			Q(z) = e^{i\nu}c z^\gamma,\quad \nu\in\Rr,\quad c>0.
		\end{equation*}
		By holomorphicity of $Q$, this holds on $\widetilde{\Cc^*}$.
		Going back to the Gauss equation in \eqref{eq:gauss-codazzi-revolution}, we have
		\begin{equation}\label{eq:equation-omega}
			\omega_{rr} + r\inv \omega_r = c^2 r^{2\gamma}e^{-2\omega}
		\end{equation}
		and a two-parameter family of solutions is given by
		\begin{equation}\label{eq:solutions-omega}
			\omega(r) = \log\left(\frac{1}{2}\left(a^2 r^{2\alpha} + b^2 r^{2\beta}\right)\right)
		\end{equation}
		where $a>0$, $b>0$ and $\alpha<\beta$ are defined by
		\begin{equation*}
			\alpha = \frac{1}{2}\left(1+\gamma - \frac{c}{ab}\right),\quad \beta = \frac{1}{2}\left(1+\gamma + \frac{c}{ab}\right).
		\end{equation*}
		With such $\omega$, the metric and Hopf differential of $f$ read as in equations \eqref{eq:metric-revolution} and \eqref{eq:hopf-revolution}.
		
		Note that all solutions to \eqref{eq:equation-omega} are given by \eqref{eq:solutions-omega}.
		Indeed, let $\omega$ be given by \eqref{eq:solutions-omega}. The function
		\begin{equation*}
			\psi\colon\Rr_+^*\times \Rr_+^*\to\Rr^2,\quad (a,b)\mapsto(\omega(1),\omega'(1))
		\end{equation*}
		is surjective, and the ODE is given by
		\begin{equation*}
			F\colon\Rr_{>0}\times \Rr^2\to\Rr^2,\quad (r,x, y)\mapsto(y, c^2 r^{2\gamma}e^{-2x} - r\inv y)
		\end{equation*}
		which is locally Lipschitz with respect to the variable $(x,y)$. By the Picard-Lindel\"of theorem, \eqref{eq:solutions-omega} gives all solutions for $a,b>0$.
\end{proof}

\subsection{LWR potential for intrinsic surfaces of revolution}\label{sec:surf-rev-generic}

Let $a,b>0$, $\alpha<\beta\in\Rr$, $\nu\in[0, 2\pi)$ and $x = (az^\alpha, b z^\beta)$. Consider the following LWR data on $\widetilde{\Sigma}$:
\begin{equation}\label{eq:data-intrinsic-surface-revolution}
	\xi = \la xx^\perp dz,\quad \Phi(1) = \I,\quad (\la_0,\la_1) = \begin{cases}
		(0, e^{i\nu}) & \text{in }\Ee^3,\\
		(-e^{i\nu}, 0) & \text{in }\Hh^3.
	\end{cases}
\end{equation}

\begin{theorem}\label{theorem:radial-metric-surfaces}
	The LWR data \eqref{eq:data-intrinsic-surface-revolution} induces an intrinsic surface of revolution with metric and Hopf differential as in \eqref{eq:metric-revolution}--\eqref{eq:hopf-revolution}. Moreover, any non-flat, minimal or CMC $1$ intrinsic surface of revolution can be obtained this way, up to an isometry and a coordinate change.
\end{theorem}
\begin{proof}
	At $\la = 0$, the potential $\xi$ is constantly vanishing, so the frame $\Phi$ is constantly the identity. Therefore, the metric of the induced immersion $f$ can be computed with theorem~\ref{theorem:fundamental-forms} and agrees with \eqref{eq:metric-revolution}.
	We deduce that $f$ is an intrinsic surface of revolution.
	Theorem~\ref{theorem:fundamental-forms} also gives the Hopf differential as in \eqref{eq:hopf-revolution} and the surface is not flat.
	
	Conversely, let $f\colon\widetilde{\Cc^*}\to\Xx^3$ be a minimal or CMC $1$, non-flat, intrinsic surface of revolution. By proposition~\ref{proposition:metric-hopf-revolution}, there exist $a,b>0$, $\alpha<\beta$ and $\nu\in\Rr$ such that the metric and Hopf differential of $f$ are given by Equations~\eqref{eq:metric-revolution}--\eqref{eq:hopf-revolution}.
	By the first part of the proof together with Bonnet's Theorem, up to a rigid motion, $f=\LWR(\Sigma, \xi, \Phi, \la_0, \la_1)$ with $\la_1-\la_0 = e^{i\nu}$.
\end{proof}

\subsection{Enneper surfaces}\label{sec:enneper}

\begin{definition}
	A minimal or CMC $1$ immersion is an \textbf{Enneper surface} if it is a non-flat, intrinsic surface of revolution that descends from the universal cover of $\Cc^*$ to a conformal immersion on $\Cc$.
\end{definition}

Let $r>0$ and $n\in\Nn_{\geq 0}$. Consider the following LWR data on $\Cc$:
\begin{equation}\label{eq:lwr-data-enneper}
	\xi = \matrix{0}{r z^n}{\la}{0}dz,\quad \Phi(0)=\I,\quad (\la_0,\la_1) = \begin{cases}
		(0,1) \text{ in } \Ee^3, \\
		(1,0) \text{ in } \Hh^3.
	\end{cases}
\end{equation}

\begin{theorem}\label{theorem:enneper}
	The LWR data \eqref{eq:lwr-data-enneper} induces an Enneper surface. Moreover, any Enneper surface can be obtained this way, up to an  isometry and coordinate change.
\end{theorem}
\begin{proof}
	We first show that the LWR data \eqref{eq:lwr-data-enneper} yields an Enneper surface. Let $f = \LWR(\Cc, \xi, \Phi, \la_0, \la_1)$. One can compute explicitly:
	\begin{equation}\label{eq:spinor-enneper}
		\Phi_0(z) = \matrix{1}{\frac{r z^{n+1}}{n+1}}{0}{1}.
	\end{equation}
	By theorem~\ref{theorem:fundamental-forms}, the metric of $f$ is
	\begin{equation}\label{eq:metric-enneper}
		ds^2 =  \left(1 + \frac{r^2\abs{z}^{2n+2}}{(n+1)^2}\right)^2\abs{dz}^2.
	\end{equation}
	The metric is rotationally invariant, so $f$ is an Enneper surface. Note that the Hopf differential of $f$ is
	\begin{equation}\label{eq:hopf-enneper}
		Qdz^2 =  (\la_1-\la_0) r z^n dz^2.
	\end{equation}
	
	We now show that any Enneper surface $f\colon\Cc\to\Xx^3$ can be obtained this way. 
	By definition, the metric of $f$ is radial and $f$ is not flat.
	By theorem~\ref{theorem:radial-metric-surfaces}, $f$ can be obtained with the LWR data \eqref{eq:data-intrinsic-surface-revolution}.
	The metric of $f$ is given by
	\begin{equation*}
		ds^2 = \left({a}^2\abs{z}^{2\alpha} + {b}^2 \abs{z}^{2\beta}\right)^2 \abs{dz}^2.
	\end{equation*}
	But $f$ is conformal at $z=0$, so $\alpha\beta=0$ and $\alpha+\beta\geq 0$.
	Because $\alpha<\beta$, we have $\alpha = 0$ and $\beta>0$.
	The Hopf differential is
	\begin{equation*}
		Qdz^2 = e^{i\nu} ab \beta z^{\beta - 1}dz^2.
	\end{equation*}
	The immersion $f$ is well-defined on $\Cc$, so there exists $n\in\Nn_{\geq 0}$ such that $\beta -1 = n$.
	Therefore, the metric and Hopf differential of $f$ are
	\begin{equation*}
		ds^2 = \left(a^2 + b^2\abs{z}^{2n+2}\right)^2\abs{dz}^2,
	\end{equation*}
	\begin{equation*}
		Qdz^2 = e^{i\nu}ab(n+1)z^{n}dz^2.
	\end{equation*}
	The immersion is not flat, so $Q\neq 0$ and $a\neq 0$. After the coordinate change
	\begin{equation*}
		w = e^{\frac{i\rho}{n+2}}a^2z
	\end{equation*}
	where $\rho\in\Rr$ is defined by
	\begin{equation*}
		e^{i\rho} := (\la_1-\la_0)e^{i\nu} = \pm e^{i\nu},
	\end{equation*}
	and after letting
	\begin{equation*}
		r = \frac{(n+1)b}{a^{2n+3}}>0,
	\end{equation*}
	one can compute that the metric and Hopf differential of $f$ read as equations \eqref{eq:metric-enneper} and \eqref{eq:hopf-enneper} in the coordinate $w$. 
	By Bonnet's theorem, $f=\LWR(\Cc, \xi, \Phi, \la_0, \la_1)$, up to a coordinate change and an isometry.
\end{proof}

\begin{corollary}
	Any Enneper surface admits an extrinsic rotational symmetry of order $n+2$.
\end{corollary}
\begin{proof}
	Let $f$ be an Enneper surface. By theorem~\ref{theorem:enneper}, $f$ is given by the LWR data \eqref{eq:lwr-data-enneper}. Let $\theta := 2 \pi/(n+2)$,  $\tau(z) := e^{i\theta}z$ and compute
	\begin{equation*}
		\tau^*\xi = \xi\cdot g,\quad g = \matrix{e^{-i\theta/2}}{0}{0}{e^{i \theta/2}},
	\end{equation*}
	so that
	\begin{equation}\label{eq:RPhig}
		\tau^*\Phi = R \Phi g
	\end{equation}
	where $R$ is holomorphic in $\la$ and independent of $z$. Evaluating \eqref{eq:RPhig} at $z=0$ and recalling that $\Phi(0) = \I$ gives $R=g\inv$. Therefore, $R\in\SU_2$ is independent of $\lambda$. By theorems~\ref{theorem:symmetries-H3} and~\ref{theorem:symmetries-E3}, $\tau$ induces a symmetry of $f$ and a direct computation shows that this symmetry is a rotation of angle $\theta$ in both ambient spaces.
\end{proof}

\subsection{Catenoids}\label{sec:catenoids}

In this section, we review catenoids in the framework of the LWR. We exhibit a family of Fuchsian potentials that induce catenoids and show that any catenoid can be obtained this way.

\begin{definition}
	A minimal or CMC 1 conformal immersion $f\colon\widetilde{\Cc^*}\to\Xx^3$ is a \textbf{catenoid} if it is a non-flat extrinsic surface of revolution: for all $t\in\Rr$,
	\begin{equation*}
		\tau_t^* f = \rR(t)\circ f
	\end{equation*}
	where $\tau_t$ is a lift of $z\mapsto e^{it}z$ and $\rR:\Rr\to\Iso(\Xx^3)$ is a 1-parameter group of rotations.
\end{definition}

\begin{remark}
	We do not assume that all catenoids are closed on $\Cc^*$. However, noting that the group $\rR$ is compact, any catenoid closes on some angular sector in $\widetilde{\Cc^*}$. With $t_0:=\min\{t>0\mid \tau_t^*f = f\}$, we define the \textbf{wrapping number} of $f$ as $r:=2\pi/t_0$.
\end{remark}

\begin{proposition}\label{prop:catenoid-potential}
	Let $\Phi$ be an LWR frame inducing a catenoid at the evaluation points $(\la_0, \la_1) = (0,1)$. Up to an LWR gauge, the LWR potential of $\Phi$ reads
	\begin{equation}\label{eq:potential-catenoid}
		\xi = Kz\inv dz,\quad K=\matrix{0}{1}{q\la + p}{0}
	\end{equation}
	where $q\in\Rr^*$ and $p>0$. Moreover, the Hopf differential of the catenoid is $qz^{-2}dz^2$ and its wrapping number is $2\sqrt{p}$.
\end{proposition}
\begin{proof}
	Let $f\colon\widetilde{\Cc^*}\to\Xx^3$ be the catenoid induced by $\Phi$. By assumption, $f$ is not flat, so there exists a Schwarz potential inducing $f$ locally (by theorem~\ref{theorem:schwarz-potential}), and up to a gauge,
	\begin{equation*}
		\xi = \matrix{0}{1}{Q\la + S}{0}dz.
	\end{equation*}
	By corollary~\ref{corrolary:schwarz-spinor}, $Qdz^2$ is the Hopf differential of $f$ and $S$ is the Schwarzian derivative of its Gauss map. Let $t\in\Rr$ and $\tau(z) = e^{it}z$. Then
	\begin{equation*}
		\tau^*\xi = \matrix{0}{1}{\tau^*Q\la + \tau^* S}{0}e^{it}dz.
	\end{equation*}
	We put this potential into its Schwarz form with a diagonal LWR gauge $g\in\SU_2$:
	\begin{equation*}
		\tau^* \xi \cdot g = \matrix{0}{1}{e^{2it}(\tau^*Q\la + \tau^* S)}{0}dz.
	\end{equation*}
	But $f$ is a catenoid, so $\tau$ induces a rotation of the surface. By theorems~\ref{theorem:symmetries-H3} and~\ref{theorem:symmetries-E3}, there exists an LWR gauge $\tilde{g}$ such that $\tau^*\xi = \xi\cdot \tilde{g}$. 
	Therefore, $\xi$ and $\tau^* \xi \cdot g$ are Schwarz potentials that lie in the same gauge class.
	By uniqueness of the Schwarz potential (theorem~\ref{theorem:schwarz-potential}), this implies that $\xi = \tau^*\xi\cdot g$, i.e.
	\begin{equation}\label{eq:tauQ}
		\tau^* Q = e^{-2it}Q,\quad \tau^*S = e^{-2it}S.
	\end{equation}
	But $Q$ is holomorphic and not identically zero, so by \eqref{eq:tauQ} $Q$ never vanishes. Therefore there exists $q\in\Cc^*$ such that $Q = qz^{-2}$. For the same reason, there exists $s\in\Cc$ such that $S=sz^{-2}$. 
	The Schwarz potential reads
	\begin{equation*}
		\xi = \matrix{0}{1}{z^{-2}(q\la+s)}{0}dz
	\end{equation*}
	and $\xi$ extends meromorphically to $\Cc^*$. With $p:=s+1/4$,
	\begin{equation}\label{eq:simple-pole-potential}
		\xi \cdot \matrix{\sqrt{z}}{0}{-\frac{1}{2 \sqrt{z}}}{\frac{1}{\sqrt{z}}} = \matrix{0}{1}{q\la +p}{0}\frac{dz}{z}.
	\end{equation}
	Up to a gauge, we can therefore assume that $\xi=K z\inv dz$ where
	\begin{equation*}
		K = \matrix{0}{1}{\la q + p}{0}
	\end{equation*}
	with $q\in\Cc^*$ and $p\in\Cc$. 
	
	We now show that $q\in\Rr^*$ and $p>0$. The frame $\Phi$ satisfies $\Phi\inv d\Phi = Kz\inv dz$, so there exists $C=(C_\la)_{\la\in\Cc}$ independent of $z$ such that $\Phi = Cz^K$. Therefore, $\tau^*\Phi = R\Phi$ with
	\begin{equation*}
		R = C\exp(i t K) C\inv.
	\end{equation*}
	Recall that for all $t$, $\tau$ induces a rotation of $f^\Xx$ with a fixed axis.
	
	In $\Hh^3$, by theorem~\ref{theorem:symmetries-H3}, $R_1\in\SU_2$ for all $t$. But this implies that $\exp(i t K_1)$ is unitarizable for all $t$, which in turn implies that $\det K_1\leq 0$. Assume by contradiction that $\det K_1 = 0$. Then $\exp(i t K_1) = \smatrix{1}{i t}{0}{1}$ and this matrix is not unitarizable by $\SL_2\Cc$-conjugation for all $t$. Therefore $\det K_1<0$, i.e. $p+q\in\Rr_{>0}$. 
	The eigenvalues of $R_0$ are $e^{\pm i t \sqrt{p}}$, so by corollary~\ref{cor:identical-immersions}, $\sqrt{p}\in\Rr^*$ and $t_0=\frac{2\pi}{2\sqrt{p}}$, i.e. the wrapping number is $2\sqrt{p}$ and  $p\in\Rr_{>0}$. Finally, $q\in\Rr^*$ because $p+q\in\Rr$ and $q\neq 0$ because a catenoid is not flat.
	
	In $\Ee^3$, $p>0$ for the same reason as in $\Hh^3$: theorem~\ref{theorem:symmetries-E3} implies that $\det K_0<0$. 
	The eigenvalues of $R_0$ are $e^{\pm i t \sqrt{p}}$, so by corollary~\ref{cor:identical-immersions}, $\sqrt{p}\in\Rr^*$ and $t_0=\frac{2\pi}{2\sqrt{p}}$.
	Using $\det K_0<0$, we note that $K$ is holomorphically diagonalizable in a neighborhood of $\la=0$ and write $R = UDU\inv$ where $U$ is independent of $t$, $U_0\in\SU_2$, and
	\begin{equation*}
		D = \matrix{e^{it\mu}}{0}{0}{e^{-it\mu}},\quad \mu^2 = q\la + p.
	\end{equation*}
	Because $U$ is independent of $t$, by proposition~\ref{prop:rigid-motion-homomorphism-E3}, $R$ induces a $1$-parameter group of rotations if and only if $D$ induces a $1$-parameter group of rotations. Consider $\phi(D)$ and $\rho(D)$ given by proposition~\ref{prop:rigid-motion-homomorphism-E3}. The linear part of $\phi$ is a rotation whose axis consists of the diagonal elements of $\Ee^3$. Therefore, $\phi(D)$ is a rotation if and only if the affine part $\rho(D) + \rho(D)^*$ is off-diagonal. By an explicit computation, 
	\begin{equation*}
		\rho(D) = \frac{itq}{2\sqrt{p}}\matrix{1}{0}{0}{-1},
	\end{equation*} 
	so $q$ is real and $q\in\Rr^*$.
\end{proof}

\begin{theorem}\label{theorem:catenoids-example}
	Let $q\in\Rr^*$ and let $p>0$. Consider the LWR data $(\Cc^*, \xi, \Phi, 0, 1)$ with $\xi$ as in \eqref{eq:potential-catenoid} and
	\begin{equation*}
		\Phi(1) := \frac{1}{\sqrt{2}}\matrix{1}{\frac{1}{\mu}}{-\mu}{1},\quad \mu := \begin{cases}
			\sqrt{p}&\text{ in }\Ee^3,\\
			\sqrt{p+q}&\text{ in }\Hh^3
		\end{cases}
	\end{equation*}
	and assume that $p+q>0$ in the case of $\Hh^3$.
	Then the induced immersion is a catenoid with Hopf differential $qz^{-2}dz^2$, wrapping number $2\sqrt{p}$ and metric
	\begin{equation}\label{eq:metric-catenoid}
		ds^2 = \frac{q^2}{4\mu^2}\left(\mu\inv\abs{z}^{2\mu }+ \mu\abs{z}^{-2\mu}\right)^2\abs{z}^{-2}\abs{dz}^2.
	\end{equation}
	
	Moreover, any catenoid can be obtained this way, up to a rigid motion and a coordinate change.
\end{theorem}
\begin{proof}
	In both cases, the LWR frame reads
	\begin{equation*}
		\Phi = Cz^K,\quad C:=\Phi(1).
	\end{equation*}
	For all $t\in\Rr$, let $\tau_t(z) = e^{it}z$. Then $\tau_t^*\Phi = R\Phi$ with
	\begin{equation*}
		R = C\exp(i t K)C\inv.
	\end{equation*} 
	
	In $\Ee^3$, the matrix $C$ is a diagonalizer of $K$ at $\la_0=0$, and thus a diagonalizer of $\exp(itK_0)$. Therefore $R_0\colon\Rr\to\SU_2$ is a $1$-parameter group. By theorem~\ref{theorem:symmetries-E3}, $\tau_t$ induces a $1$-parameter group of symmetries $\jJ_t$ on the immersion $f$. 
	Moreover, this $1$-parameter group is closed, with period $t_0=\pi/\sqrt{p}$. Therefore $\jJ$ is a $1$-parameter group of rotations with a common axis, and the induced immersion is a catenoid of wrapping number $2\sqrt{p}$.
	
	Similarly, in $\Hh^3$, the matrix $C$ is a diagonalizer of $K$ at $\la_1=1$, and thus a diagonalizer of $\exp(itK_1)$. Therefore $R_1\colon\Rr\to\SU_2$ is a $1$-parameter group. By theorem~\ref{theorem:symmetries-H3}, $\tau_t$ induces a $1$-parameter group of symmetries $\jJ_t$ on the immersion $f$. Moreover, this group is closed with period $\pi/\sqrt{p}$, so it is a $1$-parameter group of rotations with common axis. Therefore, the induced immersion is a catenoid of wrapping number $2\sqrt{p}$.
	
	The metric is given by theorem~\ref{theorem:fundamental-forms}, computing explicitly the spinors at $\la=0$ in $\Ee^3$ and $\la=1$ in $\Hh^3$:
	\begin{equation}\label{eq:spinor-catenoid}
		x = (0,i\sqrt{q}z^{-1/2}), \quad y = \frac{i\sqrt{q}}{\sqrt{2\mu}}(\mu^{-1/2}z^{-1/2 + \mu}, \mu^{1/2}z^{-1/2 -\mu}).\qedhere
	\end{equation}
	
	We now show that any catenoid can be obtained this way.
	Let $f\colon\widetilde{\Cc^*}\to\Xx^3$ be a catenoid. By theorem~\ref{theorem:immersions}, $f$ can be obtained locally via LWR at the evaluation points $(\la_0,\la_1)=(0,1)$. By proposition~\ref{prop:catenoid-potential}, up to a gauge, $\xi$ is as in \eqref{eq:potential-catenoid}. Note that the proof of proposition~\ref{prop:catenoid-potential} implies that $p+q>0$ if $\Xx^3 = \Hh^3$. 
	By proposition~\ref{prop:catenoid-potential} again, the Hopf differential of $f$ is $qz^{-2}dz^2$ and its wrapping number is $2\sqrt{p}$.
	Let $\tilde{f}$ be the catenoid given by the data of theorem~\ref{theorem:catenoids-example} with the same $q$ and $p$. Let $\Phi = Cz^K$ and $\tilde{\Phi} = \tilde{C}z^K$ be the frames for $f$ and $\tilde{f}$ respectively. Up to a rigid motion, assume that $f$ and $\tilde{f}$ share the same axis of symmetry. Then, at $\la_k$ ($k=0$ in $\Ee^3$, $k=1$ in $\Hh^3$),
	\begin{equation*}
		(CKC\inv)_ {\la_k} = (\tilde{C}K\tilde{C}\inv)_{\la_k} = \matrix{\mu}{0}{0}{-\mu}.
	\end{equation*}
	We deduce that there exists $\rho\in\Cc^*$ such that
	\begin{equation*}
		C_ {\la_k} = \matrix{\rho}{0}{0}{\rho\inv}\tilde{C}_{\la_k}.
	\end{equation*}
	One can then compute the spinor of $\Phi$ at $\la_k$:
	\begin{equation*}
		y_{\la_k} = \matrix{\rho}{0}{0}{\rho\inv}\tilde{y}_{\la_k}.
	\end{equation*}
	By \eqref{eq:spinor-catenoid}, the metrics $ds^2$ and $\widetilde{ds}^2$ agree up to a coordinate change. Moreover, the two catenoids have the same Hopf differential $qz^{-2}dz^2$. By Bonnet's Theorem, they agree up to a rigid motion. 
\end{proof}


\section{$n$-noids}\label{sec:noids}

In this section, we give a standard form for LWR potentials that induce $n$-noids. We then show how to construct generic trinoids by solving a period problem that amounts to the unitarization of a monodromy representation.

\begin{definition}
	A genus zero \textbf{$n$-noid} is a conformal, minimal or CMC $1$ immersion of the $n$-punctured sphere obtained from an LWR potential $\xi$ which is
	\begin{enumerate}
		\item of Fuchsian type: around each puncture $p_k$, there exists an LWR gauge $g_k$ such that $\xi\cdot g_k$ has a simple pole at $p_k$, and
		\item of catenoid type: the residue at $p_k$ of $\xi\cdot g_k$ has eigenvalues $\pm\sqrt{q_k\la+1/4}$. The parameter $q_k\in\Cc^*$ is the \textbf{weight} of the puncture $p_k$. 
	\end{enumerate}
\end{definition}

Our definition of a pole of catenoid type is justified by the following lemma.
\begin{lemma}\label{lemma:explicit-catenoids}
	Let $\xi = Kz\inv dz$ be an LWR potential on $\Cc^*$ where the residue $K$ is independent of $z$ and satisfies 
	\begin{equation*}
		\det K = -(q\la + p)
	\end{equation*}
	for some $q\in\Cc^*$ and $p\in\Cc$. Then there exists an LWR gauge $g$ such that
	\begin{equation*}
		\xi\cdot g = \matrix{0}{1}{q\la + p}{0}\frac{dz}{z}.
	\end{equation*}
\end{lemma}
\begin{proof}
	Write $K=A\la +B$ and let $x:=(u,v)$ be a spinor for $A$ ($x$ does not depend on $z$).
	Up to conjugation by
	\begin{equation*}
		g_0 = \matrix{0}{i}{i}{0},
	\end{equation*}
	one can assume that $v\neq 0$.
	Recall that $q\neq 0$ and compute:
	\begin{equation*}
		g_1 := \matrix{v\inv}{u}{0}{v},\quad \xi_1:=\xi\cdot g_1=\matrix{\sqrt{p}}{-q}{-\la}{-\sqrt{p}}\frac{dz}{z},
	\end{equation*}
	\begin{equation*}
		g_2:=\matrix{\sqrt{-q}}{0}{0}{\frac{1}{\sqrt{-q}}},\quad \xi_2 := \xi_1\cdot g_2 = \matrix{\sqrt{p}}{1}{ q\la}{-\sqrt{p}}\frac{dz}{z},
	\end{equation*}
	\begin{equation*}
		g_3:=\matrix{1}{0}{-\sqrt{p}}{1},\quad \xi_3:=\xi_2\cdot g_3 = \matrix{0}{1}{q\la +p}{0}\frac{dz}{z}.
	\end{equation*}
\end{proof}

\subsection{Potentials for $n$-noids}

\begin{theorem}\label{theorem:nnoid-potential}
	Any $n$-noid can be obtained with the LWR at $(\la_0,\la_1)=(0,1)$ from the potential
	\begin{equation}\label{eq:nnoid-potential}
		\xi = \matrix{0}{1}{Q\la + S}{0}dz
	\end{equation}
	where
	\begin{itemize}
		\item $Qdz^2$ is rational on $\CP^1$ with double poles at $z_1,\dots, z_n$ and no other poles. Its quadratic residues are $q_k$, the weight of $\xi$ at $z_k$. 
		\item $Sdz^2$ is rational on $\CP^1$ with double poles at the zeroes $u_k$ of $Q$ and no other poles. Its quadratic residues are
		$(n_k^2 -1)/4$ where $n_k = 1+\ord_{u_k}Q$.
	\end{itemize}
\end{theorem}
\begin{proof}
	By theorem~\ref{theorem:schwarz-potential}, it suffices to compute the Hopf differential and the Schwarzian derivative of the Gauss map of a given $n$-noid $f\colon \Sigma\to\Xx^3$ where $\Sigma = \CP^1 \backslash\{z_1,\dots,z_n\}$.
	
	Away from the punctures $z_k$, the Hopf differential is a holomorphic quadratic differential and the Gauss map is a meromorphic function, so both $Q$ and $S$ are rational on $\CP^1$. 
	The Hopf differential is holomorphic, so has no pole in $\Sigma$. Assume that $G$ has a pole or zero of order $n\neq 0$ at some $u\in\Sigma$. Then a computation shows that its Schwarzian derivative has a double pole at $u$ with quadratic residue $(n^2-1)/4$. Moreover, for the surface to be immersed at $u$, $Q$ must have a zero at $u$ of order $n-1$ (this is in fact a sufficient condition). 
	
	Now we compute the behavior of $Q$ and $S$ at the punctures $z_k$. The potential $\xi$ is of Fuchsian type, so it is by definition gauge-equivalent to a potential $\eta$ of catenoid type at $z_k$ with residue 
	\begin{equation*}
		K = \matrix{0}{1}{\mu^2}{0},\quad \mu^2 = \la q_k + \frac{1}{4},\quad q_k\in\Cc^*.
	\end{equation*}
	With remark~\ref{remark:hopf} applied to $\eta$, one can compute that $Q$ admits a double pole at $z_k$ with quadratic residue $q_k$. By computing the gauge putting $\eta$ into its Schwarz form (using theorem~\ref{theorem:schwarz-potential}), one can show that $S$ has at most a simple pole at $z_k$ (its quadratic residue vanishes because $\mu_0^2 = 1/4$). But the $n$-noid $f$ closes around $z_k$, so by theorem~\ref{theorem:closing-conditions}, the monodromy of $\Phi_0$ around $z_k$ is $M_0=\pm\I$, and this happens only if the residue of $S$ vanishes. Therefore, $S$ is holomorphic at $z_k$.
\end{proof}

\subsection{Unitarizability on the three-punctured sphere}

\begin{definition}~
	\begin{itemize}
		\item A subset $\xX\subset\SL_2\Cc$ (resp. $\mathrm{sl}_2\Cc$) is \textbf{reducible} if there exists $\ell\in\CP^1$ which is an eigenline of $X$ for all $X\in\xX$.
		\item A subset $\xX\subset\SL_2\Cc$ (resp. $\mathrm{sl}_2\Cc$) is \textbf{unitarizable} if there exists $C\in\SL_2\Cc$ such that $CXC\inv \in\SU_2$ (resp. $\su_2$)	for all $X\in\xX$.
		\item $\nu\in\Cc$ is a \textbf{logarithmic eigenvalue} of $M\in\SL_2\Cc$ if $e^{2\pi i \nu}$ is an eigenvalue of $M$. The logarithmic eigenvalues of $M$ are defined up to $\nu\mapsto -\nu$ and $\nu\mapsto\nu+1$, and can be \textbf{normalized} uniquely so that $\Re\nu\in[0, 1/2]$.
	\end{itemize}
\end{definition}

\begin{theorem}
	Let $M_0, M_1, M_2\in\SL_2\Cc$ satisfy $M_0 M_1 M_2 = \I$. Let $\nu_0, \nu_1,\nu_2\in\Cc$ be corresponding logarithmic eigenvalues. Then
	\begin{itemize}
		\item $\{M_0, M_1, M_2\}$ is reducible if and only if $ \nu_0\pm \nu_1 \pm \nu_2 \in\Zz$ for some choice of signs.
		\item $\{M_0, M_1, M_2\}$ is irreducible and unitarizable if and only if $\nu_0, \nu_1, \nu_2\in\Rr$ and, when normalized to $[0,1/2]$, satisfy the spherical triangle inequalities:
		\begin{equation*}
			\nu_0<\nu_1 + \nu_2,\quad \nu_1<\nu_0 + \nu_2,\quad \nu_2<\nu_0 + \nu_1 ,\quad \nu_0 + \nu_1 + \nu_2<1.
		\end{equation*}
	\end{itemize}
\end{theorem}
\begin{proof}
	Let 
	\begin{equation*}
		t_k := \frac{1}{2}\tr M_k,\quad k=0,1,2,
	\end{equation*}
	and
	\begin{equation*}
		\varphi := 1-t_0^2-t_1^2-t_2^2 + 2t_0t_1t_2.
	\end{equation*}
	By \cite{goldman}, $\{M_0, M_1, M_2\}$ is reducible if and only if $\varphi=0$. Moreover, $\{M_0, M_1, M_2\}$ is irreducible and unitarizable if and only if $t_0, t_1, t_2\in(-1,1)$ and $\varphi > 0$. 
	Let $e_k := e^{2\pi i \nu_k}$ so that
	\begin{equation*}
		t_k = \frac{e_k + e_k\inv}{2},\quad k=0,1,2.
	\end{equation*}
	Then
	\begin{equation*}
		4e_0^2e_1^2e_2^2 \varphi = (e_0e_1e_2-1)(e_0e_1-e_2)(e_0e_2-e_1)(e_1e_2-e_0).
	\end{equation*}
	So $\varphi = 0$ if and only if $ \nu_0\pm \nu_1 \pm \nu_2 \in\Zz$ for some choice of signs. This proves the first point. 
	
	To prove the second point, note that $t_k\in(-1,1)$ if and only if $\nu_k\in\Rr$. Assume that $\nu_0, \nu_1, \nu_2$ are normalized. Then $\varphi=0$ if and only if $(\nu_0, \nu_1, \nu_2)$ lies on the boundary of the tetrahedron $\tT\subset [0,1/2]^3$ defined by
	\begin{equation*}
		(\nu_0 = \nu_1+\nu_2) \quad\text{or}\quad (\nu_1 = \nu_0+\nu_2)\quad \text{or}\quad (\nu_2 = \nu_0+\nu_1) \quad\text{or}\quad (\nu_0 + \nu_1+\nu_2 =1).
	\end{equation*} 
	Using the continuity of $\varphi$, one can check that $\varphi>0$ if and only if $(\nu_0, \nu_1, \nu_2)$ lies in the interior of $\tT$, that is, $\nu_0, \nu_1, \nu_2$ satisfy the spherical triangle inequalities.
\end{proof}

\begin{theorem}\label{theorem:unitarization-E3}
	Let $A_0, A_1, A_2\in\mathrm{sl}_2\Cc$ satisfy $A_0 + A_1 + A_2 = 0$. Let $ia_0, ia_1,ia_2\in\Cc$ be corresponding eigenvalues. Then
	\begin{itemize}
		\item $\{A_0, A_1, A_2\}$ is reducible if and only if $ a_0\pm a_1 \pm a_2 =0$ for some choice of signs.
		\item $\{A_0, A_1, A_2\}$ is irreducible and unitarizable if and only if $a_0, a_1, a_2\in \Rr$ and, when normalized to $\Rr_+$, satisfy the Euclidean triangle inequalities:
		\begin{equation*}
			{a_0}<{a_1} + {a_2},\quad {a_1}<{a_0} + {a_2},\quad {a_2}<{a_0} + {a_1}.
		\end{equation*}
	\end{itemize}
\end{theorem}
\begin{proof}
	Let 
	\begin{equation*}
		\varphi := \det [A_1, A_2].
	\end{equation*}
	We first show that $\{A_0, A_1, A_2\}$ is reducible if and only if $\varphi=0$. 
	If $\{A_0, A_1, A_2\}$ is reducible, then they share an eigenline and this eigenline is in the kernel of the commutator $[A_1, A_2]$, which implies that $\varphi = 0$. Conversely, if $\varphi=0$, then the commutator $[A_1, A_2]$ is nilpotent and can therefore be written $[A_1,A_2] = xx^\perp$ for some $x\in\Cc^2$. If $x=0$, then $A_1$ and $A_2$ commute, so they share an eigenline and $\{A_0, A_1, A_2\}$ is reducible. Assume that $x\neq 0$. Considering the bilinear extension of equations~\eqref{eq:lorentzian-inner-product} and \eqref{eq:cross-product-E3}, one has
	\begin{equation*}
		0 = \scal{A_1, [A_1, A_2]} =\scal{A_1, xx^\perp} = \det(A_1x, x).
	\end{equation*}
	Therefore, $A_1x\in \scal{x}$ and $x$ is an eigenvector of $A_1$. Similarly, $x$ is an eigenvector of $A_2$. Therefore $\{A_0, A_1, A_2\}$ is reducible. We have proved that $\{A_0, A_1, A_2\}$ is reducible if and only if $\varphi=0$.
	
	Now a computation gives
	\begin{equation*}
		\varphi = 4(\scal{A_1, A_1}\scal{A_2, A_2} - \scal{A_1, A_2}^2).
	\end{equation*}
	By the polarization identity, and using that $A_1+A_2 = -A_0$,
	\begin{equation*}
		\varphi = 4\norm{A_1}^2\norm{A_2}^2 - (\norm{A_0}^2 - \norm{A_1}^2 - \norm{A_2}^2)^2,
	\end{equation*}
	where $\norm{A_k}^2 = \scal{A_k, A_k} = -\det A_k$. But $\det A_k = a_k^2$, so
	\begin{equation*}
		\varphi = (a_0+a_1+a_2)(a_0+a_1-a_2)(a_0-a_1+a_2)(-a_0+a_1+a_2).
	\end{equation*}
	Hence, $\varphi = 0$ if and only if $ a_0\pm a_1 \pm a_2 =0$ for some choice of signs, and the first point is proved.
	
	To prove the second point, we show that $\{A_0, A_1, A_2\}$ is irreducible and unitarizable if and only if $a_0, a_1, a_2\in \Rr$ and $\varphi > 0$. Let
	\begin{equation*}
		T := \begin{pmatrix}
			\scal{A_1, A_1} & \scal{A_1, A_2} & 0 \\
			\scal{A_1, A_2} & \scal{A_2, A_2} & 0 \\
			0 & 0 & \scal{A_1, A_1}\scal{A_2, A_2} - \scal{A_1, A_2}^2
		\end{pmatrix}.
	\end{equation*}
	If $\{A_0, A_1, A_2\}$ is irreducible and unitarizable, then $a_0, a_1, a_2\in \Rr$ and $T$ is real symmetric positive definite, so $\varphi>0$. Conversely, assume that $a_0, a_1, a_2\in \Rr$ and that $\varphi>0$. Then $T$ is real symmetric positive definite. Let $T = U^T U$ be the Cholesky decomposition of $T$, with $U$ real upper triangular. Let $S\in\SO_3\Cc$ such that $U=SV$ where $V=(A_1, A_2, A_1\times A_2)$. 
	Let $C\in\SL_2\Cc$ be a lift of $S$ given by the double covering $\SL_2\Cc\to\SO_3\Cc$ (see exercise 7.17 in \cite{fulton2013representation}).
	Then $C$ is a unitarizer of $\{A_0, A_1, A_2\}$. Therefore, $\{A_0, A_1, A_2\}$ is irreducible and unitarizable if and only if $a_0, a_1, a_2\in \Rr$ and $\varphi > 0$, which is equivalent to  $a_0, a_1, a_2\in \Rr$ and $\abs{a_0}, \abs{a_1}, \abs{a_2}$ satisfying the Euclidean triangle inequalities.
\end{proof}

\subsection{Trinoids}

References for the classification of trinoids in $\Ee^3$ are \cite{katoumeharayamada2000}, \cite{lopez1992classification} and in $\Hh^3$,  \cite{umehara2000metrics},\cite{bobenko-springborn-spinors}. Here, we outline the strategy to obtain closed trinoids by finding a unitarizer of a monodromy representation on the $3$-punctured sphere.

\begin{corollary}\label{corollary:trinoid-potential}
	Any trinoid can be obtained from the LWR with the data
	\begin{equation*}
		\Sigma = \CP^1\backslash\{0,1,\infty\},\quad \xi = \matrix{0}{1}{Q\la + S}{0}dz,\quad (\la_0, \la_1) = (0,1)
	\end{equation*}
	where
	\begin{equation}\label{eq:qs-trinoid}
		Q = \frac{q_0 + (q_1-q_0-q_2)z + q_2z^2}{z^2 (z-1)^2},
		\quad S = \frac{\frac{3}{4}(u_1-u_0)^2}{(z-u_0)^2 (z-u_1)^2},
	\end{equation}
	$q_0,q_1,q_2\in\Cc$ and $u_1\neq u_2\in\Cc$ are the zeroes of $Q$.
\end{corollary}
\begin{proof}
	The function $Q$ is as in \eqref{eq:qs-trinoid} because of theorem~\ref{theorem:nnoid-potential}. 
	The total order of the quadratic differential $Sdz^2$ is $-4$, and by theorem~\ref{theorem:nnoid-potential}, $S$ is rational on $\CP^1$ with double poles at the zeroes of $Q$ and no other poles. Therefore, $Q$ admits two distinct zeroes $u_1,u_2\in\Cc$, each of order $1$. By theorem~\ref{theorem:nnoid-potential} again, the quadratic residue of $Sdz^2$ at $u_k$ is $\frac{3}{4}$, hence the form of $S$ in \eqref{eq:qs-trinoid}.
\end{proof}

\begin{proposition}\label{prop:eigenvalues-trinoid}
	Let $\Phi$ be an LWR frame with LWR potential $\xi$ as in corollary~\ref{corollary:trinoid-potential} and let $(M_0, M_1, M_2)$ be the monodromy of $\Phi$ around some loop enclosing $(0, 1, \infty)$ with index $1$.
	\begin{enumerate}
		\item The eigenvalues of $M_k$ at $\la_1=1$ are $e^{\pm 2\pi i \nu_k}$ where
		\begin{equation}\label{eq:nu-k}
			\nu_k = \frac{1}{2} - \sqrt{q_k+1/4}.
		\end{equation}
		\item The eigenvalues of $\dot{M}_k$ at $\la_0=0$ are $\pm\dot{\nu}_k$ where
		\begin{equation}\label{eq:dot-nu-k}
			\dot{\nu}_k = 2\pi i q_k.
		\end{equation}
	\end{enumerate}
\end{proposition}
\begin{proof}
	First note that these eigenvalues do not depend upon the choice of an initial condition for $\Phi$, nor do they depend on the chosen loop.
	Around each singularity, there exists a gauge $g$ such that the potential $\xi\cdot g$ is a Fuchsian potential whose residue $K$ has eigenvalues $\pm\mu$ where
	\begin{equation*}
		\mu^2 = \la q_k + \frac{1}{4}.
	\end{equation*}
	Noting that the gauge $g$ itself has monodromy $-\I$, we deduce that the monodromy of $\Phi$ has the same eigenvalues as $-\exp(2\pi i K)$. This proves \eqref{eq:nu-k}.
	To compute the eigenvalues of $\dot{M}_k$, note that $K$ is diagonalizable in a neighborhood of $\la_0=0$. 
	Therefore, with
	\begin{equation*}
		N_k := \matrix{e^{2\pi i \mu_k}}{0}{0}{e^{-2\pi i \mu_k}},\quad \mu_k^2 = \la q_k + \frac{1}{4},
	\end{equation*}
	there exists $C=C_\la$ holomorphic for $\la$ in a neighborhood of $\la_0=0$ such that
	\begin{equation*}
		M_k = CN_kC\inv.
	\end{equation*}
	Using that $M_k=-\I$ at $\la_0=0$, we get
	\begin{equation*}
		(\dot{M}_k)_{\la_0} = C_{\la_0}(\dot{N}_k)_{\la_0}C_{\la_0}\inv
	\end{equation*}
	and that the eigenvalues of $\dot{M}_k$ and the eigenvalues of $\dot{N}_k$ are the same at $\la_0=0$. This implies \eqref{eq:dot-nu-k}.
\end{proof}

\begin{definition}\label{definition:irreducible-trinoid}
	An \textbf{irreducible trinoid} is a trinoid induced by an LWR frame $\Phi$ whose monodromy matrices $M_0, M_1, M_2$ satisfy the following condition:
	\begin{itemize}
		\item for minimal trinoids in $\Ee^3$: $\{\dot{M}_0, \dot{M}_1, \dot{M}_2\}_{\la_0}$ is irreducible,
		\item for CMC 1 trinoids in $\Hh^3$: $\{M_0, M_1, M_2\}_{\la_1}$ is irreducible.
	\end{itemize}
\end{definition}

\begin{theorem}
	Irreducible trinoids are parametrized by their weights.
\end{theorem}
\begin{proof}
	We start with Euclidean space $\Ee^3$. Let $\tT$ be the set of irreducible minimal trinoids (modulo rigid motions) and let $\qQ$ be the set of triples $(q_0, q_1, q_2)\in\Rr^3$ whose absolute values satisfy the Euclidean triangle inequalities. 
	Let $\wW\colon\tT\to\qQ$ be the map that associates to a trinoid its weights. We show that $\wW$ is a bijection.
	
	The map $\wW$ is well-defined: by corollary~\ref{corollary:trinoid-potential}, the weights of a given trinoid are determined by its Hopf differential, which is invariant under rigid motions.
	Moreover, we need to show that $\wW(\tT)\subset\qQ$: let $f=\LWR(\Sigma, \xi, \Phi, \la_0, \la_1)$ be a trinoid given by corollary~\ref{corollary:trinoid-potential} and let $M_0, M_1, M_2$ be the monodromies of $\Phi$ around the poles $0$, $1$ and $\infty$ respectively. By definition~\ref{definition:irreducible-trinoid} and by theorem~\ref{theorem:closing-conditions}, $\{\dot{M}_0, \dot{M}_1, \dot{M}_2\}_{\la_0}$ is irreducible and unitary. By proposition~\ref{prop:eigenvalues-trinoid}, the eigenvalues of $(\dot{M}_k)_{\la_0}$ are $\pm2\pi i q_k$. By theorem~\ref{theorem:unitarization-E3}, $(q_0, q_1, q_2)\in\qQ$. Note that the Hopf differential $Qdz^2$ given by corollary~\ref{corollary:trinoid-potential} has two distinct zeroes if and only if $\delta(q_1,q_2,q_3)\neq 0$ where
	\begin{equation*}
		\delta(q_0, q_1, q_2) := q_0^2 +q_1^2 +q_2^2 - 2q_0q_1 - 2 q_0q_2 - 2q_1q_2,
	\end{equation*}
	but $\qQ\cap\delta\inv\{0\}=\varnothing$.
	
	The map $\wW$ is surjective: let $q=(q_0, q_1, q_2)\in\qQ$. Let $(\Sigma, \xi, \Phi, \la_0, \la_1)$ be LWR data as in corollary~\ref{corollary:trinoid-potential}. Since $q\in\qQ$, proposition~\ref{prop:eigenvalues-trinoid} and theorem~\ref{theorem:unitarization-E3} ensure that $\{\dot{M}_0, \dot{M}_1, \dot{M}_2\}_{\la_0}$ is irreducible and unitarizable. Let $C$ be a unitarizer. Then the LWR frame $C\Phi$ induces an irreducible trinoid.
	
	The map $\wW$ is injective: let $f_1$ and $f_2$ be two irreducible trinoids with the same weights. By corollary~\ref{corollary:trinoid-potential}, $f_1$ and $f_2$ can be obtained via the LWR from the same potential and the same evaluation points. Therefore, the corresponding LWR frames $\Phi_1$ and $\Phi_2$ only differ by an initial condition: there exists $C=(C_\la)_{\la\in\Cc}$ such that $\Phi_2 = C\Phi_1$. With $M$ and $N$ the monodromies of $\Phi_1$ and $\Phi_2$ respectively, this implies that $\{\dot{M}_0, \dot{M}_1, \dot{M}_2\}_{\la_0}$ and $\{\dot{N}_0, \dot{N}_1, \dot{N}_2\}_{\la_0}$ are conjugated by $C_{\la_0}$ (see theorem~\ref{theorem:closing-conditions}). We deduce that $C_{\la_0}\in\SU_2$. By theorem~\ref{theorem:rigid-motions}, $f_1=f_2$ up to a rigid motion.
	
	The case of $\Hh^3$ is treated similarly. Let $\tT$ be the set of irreducible CMC 1 trinoids (modulo rigid motions). The relevant eigenvalues are now given by \eqref{eq:nu-k} in place of \eqref{eq:dot-nu-k}.
	Let $\vV$ be the set of triples $(\nu_0, \nu_1, \nu_2)\in\Rr^3$ whose absolute values satisfy the spherical triangle inequalities, let $\Delta := \delta\inv(\{0\})$ and let 
	\begin{equation*}
		\rho\colon(-\infty, 1/2)\to(-1/4, +\infty),\quad (\nu_0, \nu_1, \nu_2)\mapsto \left(\nu_k(\nu_k-1)\right)_{k=1,2,3}
	\end{equation*}
	and let $\qQ := \rho(\vV)\backslash\Delta$. One can show, with the same arguments as in Euclidean space, that the map $\wW\colon\tT\to\qQ$ which associates its weights to a trinoid is a bijection.
\end{proof}


\section{Simple factor dressing}\label{sec:simple-factor-dressing}

Darboux transformations of minimal and CMC $1$ surfaces are constructed in \cite{tenenblat2003, tenenblat2012}. The same type of transformations for non-minimal and CMC $H>1$ surfaces are described via a dressing action on the DPW data in \cite{cho2019simple}. We aim at obtaining the surfaces of \cite{tenenblat2003, tenenblat2012} using the methods of \cite{cho2019simple} in the framework of the LWR. 

We first define a simple factor dressing action on LWR frames and potentials, and show that simple factor dressing preserves the Hopf differential. We then give a way to control the monodromy of the dressed surface and give examples obtained from a catenoid.

\begin{figure}[h]
	\centering
	\begin{subfigure}{0.49\textwidth}
		\includegraphics[width=\linewidth]{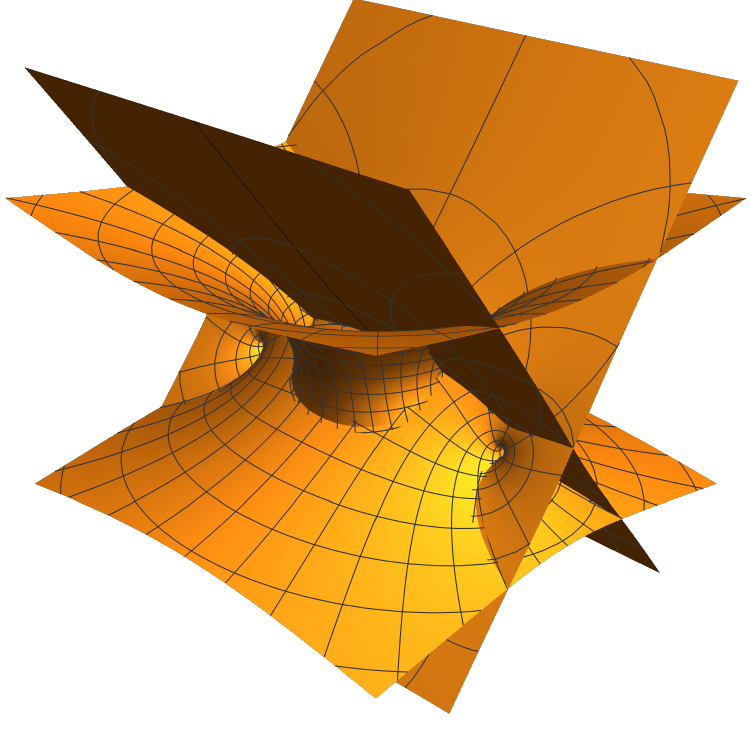}
		\caption{Full surface}
		\label{fig:full}
	\end{subfigure}
	\hfill
	\begin{subfigure}{0.49\textwidth}
		\includegraphics[width=\linewidth]{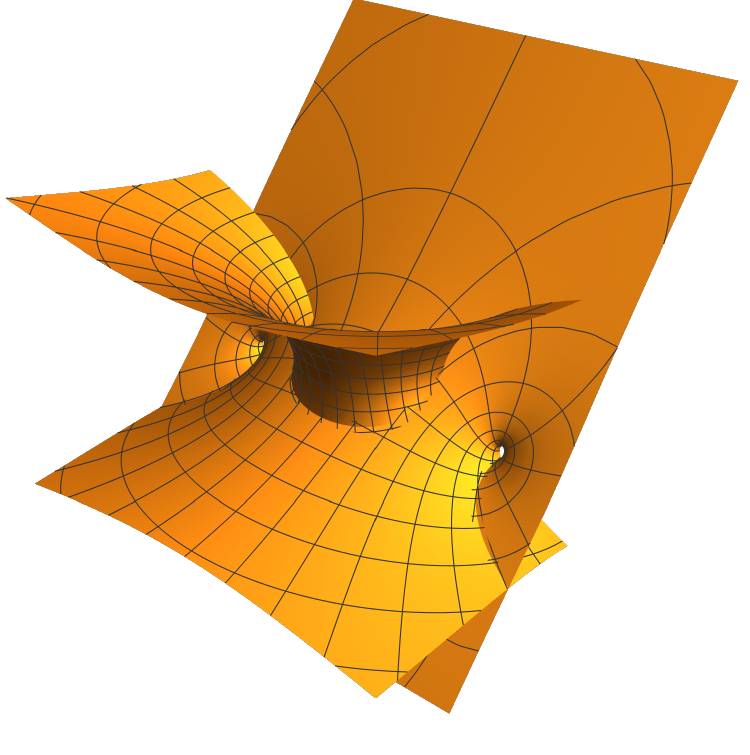}
		\caption{Half view}
		\label{fig:cut}
	\end{subfigure}
	\caption{Simple factor dressing of the catenoid in $\Ee^3$, as in example~\ref{ex:dressed-catenoid2} with $(p,q)=(\frac{1}{4},1)$ and $(u,\ell) = (2,2)$.}
	\label{fig:dressed-catenoid2}
\end{figure}

\subsection{Dressed frames}

\begin{definition}
	Let $\alpha\in\Cc$ and $\ell\neq m\in\CP^1$. The \textbf{simple factor} $\SF(\alpha, \ell, m)$ is the $\la$-family of 2-by-2 matrices
	\begin{equation*}
		\SF(\alpha, \ell, m) := S\Delta S\inv
	\end{equation*}
	where
	\begin{equation*}
		\Delta := \matrix{\la-\alpha}{0}{0}{1}
	\end{equation*}
	and $S\in\GL_2\Cc$ is independent of $\la$ and has its first column vector in $\ell$ and its second column vector in $m$. 
\end{definition}

\begin{remark}~
	\begin{itemize}
		\item Simple factors are well-defined: multiplying the columns of $S$ by scalars amounts to multiplying $S$ on the right by a diagonal matrix, which commutes with $\Delta$.
		\item $\SF$ is injective because a 2-by-2 diagonalizable matrix is determined by its ordered pairs of eigenvalues and eigenlines.
		\item Permuting the diagonal entries of $\Delta$ in $\SF(\alpha, \ell, m)$ defines the simple factor $\SF(\alpha, m, \ell)$. Thus, for all $\alpha$, $\ell$ and $m$,
		\begin{equation}\label{eq:inverse-simple-factor}
			\SF(\alpha, \ell, m)\inv = (\la-\alpha)\inv \SF(\alpha, m, \ell).
		\end{equation}
	\end{itemize}
\end{remark}

\begin{theorem}\label{theorem:simple-factor-dressing}
	Let $\Phi$ be an LWR frame and assume that $\Phi$ is not flat. For any $z$-independent simple factor $g$, there exists a unique $z$-dependent simple factor $h$ such that
	\begin{equation}\label{eq:simple-factor-dressing}
		g\#\Phi := g\Phi h\inv
	\end{equation}
	is an LWR frame.
\end{theorem}
\begin{proof}
	Write $g=\SF(\alpha, \ell, m)$.
	We first show that the simple factor $h$ is necessarily given by
	\begin{equation}\label{eq:simple-factor-h}
		h = \SF(\alpha, x, y),\quad x = \Phi_\alpha\inv \ell,\quad y=\spann w.
	\end{equation}
	where $w$ is a spinor for $A$ in the potential $\xi = (A\la+B)dz$ of $\Phi$.
	
	Let $h=\SF(\beta, x, y)$ such that $\hat{\Phi}:=g\Phi h\inv$ is an LWR frame. Then $\beta = \alpha$ because $\det \hat{\Phi} = 1$. 
	Moreover, $\hat{\Phi}$ is holomorphic at $\la=\alpha$, so on the one hand,
	\begin{equation*}
		(\la-\alpha)g\Phi h\inv \tendsto{\la\to\alpha}0.
	\end{equation*}
	On the other hand, by \eqref{eq:inverse-simple-factor},
	\begin{equation*}
		(\la-\alpha)h\inv = \SF(\alpha, y, x)
	\end{equation*}
	so
	\begin{equation*}
		(\la-\alpha)g\Phi h\inv x \tendsto{\la\to \alpha} g_\alpha \Phi_\alpha x.
	\end{equation*}
	Therefore, in order for $\hat{\Phi}$ to be holomorphic at $\la=\alpha$, it is necessary that $x=\Ker g_\alpha\Phi_\alpha$, i.e. $x=\Phi_\alpha\inv \ell$.
	To determine $y$, let $\xi=\Phi\inv d\Phi$ and $\hat{\xi} = \xi\cdot h\inv$. The frame $\hat{\Phi}$ is LWR, so $\hat{\xi}$ is an LWR potential. Writing $h=S\Delta S\inv$, $S$ is $\la$-independent, so both $\tilde{\xi}:= \xi\cdot (S\Delta\inv)$ and $\xi\cdot S$ are LWR potentials. But $\Delta$ is $z$-independent, so
	\begin{equation*}
		\tilde{\xi} = \Delta (\xi\cdot S) \Delta\inv.
	\end{equation*}
	With $\xi\cdot S = \matrix{a}{b}{c}{-a}$,
	\begin{equation}\label{eq:xitilde}
		\tilde{\xi} = \matrix{a}{b(\la-\alpha)}{c(\la-\alpha)\inv}{-a}.
	\end{equation}
	Therefore, $c_\alpha = 0$. Differentiating at $\la=\alpha$ gives
	\begin{equation*}
		\dot{\tilde{\xi}} = \matrix{\dot{a}_\alpha}{b_\alpha}{0}{-\dot{a}_\alpha}.
	\end{equation*}
	This matrix is nilpotent, so $\dot{a}_\alpha = 0$. By linearity of $\xi\cdot S$, one can write
	\begin{equation}\label{eq:xiS}
		\xi\cdot S = \matrix{\hat{a}}{\hat{b}}{(\la-\alpha)\hat{c}}{-\hat{a}}
	\end{equation}
	with $\hat{a},\hat{b},\hat{c}$ independent of $\la$. With $\xi = (A\la+B)dz$, this implies that
	\begin{equation*}
		SAS\inv = \matrix{0}{0}{\hat{c}}{0}\quad \text{and}\quad SBS\inv = \matrix{\hat{a}}{\hat{b}}{-\alpha\hat{c}}{-\hat{a}}.
	\end{equation*}
	Therefore, with $A=ww^\perp$, $w\in y$, and this determines $h=\SF(\alpha, \Phi_\alpha\inv \ell, y)$ uniquely.
	
	To show existence, let $\Sigma$ be the Riemann surface on which $\xi$ is defined and let $\tilde{\Sigma}$ be its universal cover, so that $\Phi$ is defined on $\tilde{\Sigma}$. Let $h$ as in~\eqref{eq:simple-factor-h}. Then $h$ is defined on $\tilde{\Sigma}\backslash\sS$ where
	\begin{equation}\label{eq:extra-singularities}
		\sS = \{z\in\tilde{\Sigma}\mid \Phi_\alpha w \in \ell\}.
	\end{equation}
	By meromorphicity of $\Phi$ and $w$, $\sS$ is either a set of isolated points or the entire $\tilde{\Sigma}$. Suppose that $\sS=\tilde{\Sigma}$. Then
	\begin{equation*}
		\det(\Phi w, (\Phi w)_z) \tendsto{\la\to \alpha}0
	\end{equation*} 
	because $\ell$ is constant in $z$. By theorem~\ref{theorem:fundamental-forms}, this implies that $Q=0$, and that contradicts the non-flatness of $\Phi$. Therefore, $\sS$ is a set of isolated points in $\tilde{\Sigma}$, and $h$ is well-defined on $\tilde{\Sigma}\backslash\sS$. Therefore, the frame $\hat{\Phi}$ defined by \eqref{eq:simple-factor-dressing} is a frame on $\tilde{\Sigma}\backslash\sS$, and the uniqueness part of this proof shows that it is an LWR frame.
\end{proof}

\begin{definition}
	The map $\Phi\mapsto g\#\Phi$ of \eqref{eq:simple-factor-dressing} is the \textbf{simple factor dressing} of $\Phi$ by $g$. It descends, via the LWR, to a map on the induced immersions, also called simple factor dressing.
\end{definition}

\begin{corollary}\label{corollary:hopf-simple-factor-dressing}
	Simple factor dressing preserves the Hopf differential.
\end{corollary}
\begin{proof}
	Let $Qdz^2$ and $\tilde{Q}dz^2$ be the Hopf differentials induced by $\Phi$ and $\tilde{\Phi}:=g\#\Phi=g\Phi h\inv$ respectively.
	The LWR potential of $\tilde\Phi$ is $\tilde\xi:=\xi\cdot h\inv$. By theorem~\ref{theorem:simple-factor-dressing}, $h$ is a simple factor, so $h=S\Delta S\inv$. Therefore, $\tilde{\xi} = \xi\cdot(S\Delta\inv S\inv)$, and we have
	\begin{equation*}
		\hat{\xi} := \tilde{\xi}\cdot S = (\xi\cdot S)\cdot \Delta\inv.
	\end{equation*}
	Let $\hat{\Phi}:=\tilde{\Phi}S$ be an LWR frame for $\tilde{\xi}$ and let $\hat{Q}dz^2$ be the Hopf differential induced by $\hat{\Phi}$. On the one hand, $S$ is an LWR gauge, so $\hat{Q} = \tilde{Q}$. On the other hand, assuming without loss of generality that $\xi\cdot S$ is in its Schwarz form (theorem~\ref{theorem:schwarz-potential}), with $Qdz^2 = (\la_1-\la_0)qdz^2$,
	\begin{equation*}
		\hat{\xi} = \matrix{0}{1}{q\la +s}{0}\cdot\matrix{(\la-\alpha)\inv}{0}{0}{1}dz = \matrix{0}{\la-\alpha}{\frac{q\la+s}{\la-\alpha}}{0}dz.
	\end{equation*}
	But $\hat{\xi}$ is an LWR potential, so $q\alpha + s=0$. Therefore,
	\begin{equation*}
		\hat{\xi} = \matrix{0}{\la-\alpha}{q}{0}dz.
	\end{equation*}
	By remark~\ref{remark:hopf}, $\hat{Q}=Q$.
\end{proof}

\begin{theorem}
	Let $\Phi$ be a non-flat LWR frame with monodromy $M$. Let $g=\SF(\alpha, \l, m)$. If $\l$ is an eigenline of $M_\alpha$, then the monodromy $\hat{M}$ of $g\#\Phi$ is
	\begin{equation*}
		\hat{M} = gMg\inv.
	\end{equation*}
\end{theorem}
\begin{proof}
	Let $\xi$ be the potential of $\Phi$ and let $w$ be the associated spinor. Assume that $\xi$ is defined on $\Sigma$ and let $\tilde{\Sigma}$ be its universal cover.
	Let $h=\SF(\alpha, x, y)$ with $x=\Phi_\alpha\inv \ell$ and $y=\spann w$ so that the dressed frame reads $\hat{\Phi} = g\Phi h\inv$ as in theorem~\ref{theorem:simple-factor-dressing}. 
	Note that $h$ has no monodromy around the extras singularities defined by \eqref{eq:extra-singularities}. 
	Let $\tau$ be a deck transformation of $\tilde{\Sigma}\to\Sigma$ and let $M$ be the corresponding monodromy matrix for $\Phi$.
	The potential $\xi$ is well-defined on $\Sigma$, so $\tau^*y = y$. Moreover,
	\begin{equation*}
		\tau^*x = \tau^* \Phi_\alpha\inv \ell 
		= \Phi_\alpha\inv M_\alpha\inv \ell 
		= \Phi_\alpha\inv l 
		= x
	\end{equation*}
	because $\ell$ is an eigenline of $M_\alpha$. Therefore $\tau^* h = h$, and because $g$ is $z$-independent,
	\begin{equation*}
		\tau^* \hat{\Phi} = g (\tau^*\Phi) h\inv = g M \Phi h\inv = gMg\inv \hat{\Phi}.\qedhere
	\end{equation*}
\end{proof}

\subsection{Example: dressed catenoids}

\begin{theorem}
	For any catenoid, there exists a two-real-parameter family of simple factor dressings, each one admitting a discrete rotational symmetry and an extra end on every fundamental piece. The first parameter $u>0$ determines the angle of the rotation, while the second parameter $\ell\in\RP^1$ determines the location of the end. The dressed surface closes over a cover of the catenoid's domain if and only if $u\in\Qq$. In this case, $u=n/r$ where $r$ is the wrapping number and $n$ is the number of extra ends.
\end{theorem}
\begin{proof}
	We first study the Euclidean case. Let $f$ be a catenoid in $\Ee^3$. By theorem~\ref{theorem:catenoids-example}, $f$ is obtained via the LWR with the data of theorem~\ref{theorem:catenoids-example}. Thus, $\Phi = C z^K$ with $\det K = -(q\la + p)$ with $q\in\Rr^*$ and $p>0$. Let
	\begin{equation*}
		u>0,\quad u\neq 1,\quad \ell\in\RP^1.
	\end{equation*}
	Write $\det K_\la = -\mu_\la^2$ with $\mu_0>0$. Let $\alpha:=p(u^2-1)/q\in\Rr^*$ so that $\mu_\alpha=u\sqrt{p}>0$ and $\mu_\alpha/\mu_0=u>0$. Let $g := \SF(\alpha, \ell, m)$ for any $m\in\CP^1$. Define
	\begin{equation*}
		\hat{\Phi} := \hat{g}\Phi h\inv,\quad \hat{g} :=
		\sqrt{\det g_0}g_0\inv g.
	\end{equation*}
	where $h$ is defined via $g\#\Phi = g\Phi h\inv$ as in theorem~\ref{theorem:simple-factor-dressing}.
	Let $\hat{f}$ be the immersion induced by $\hat{\Phi}$ at the evaluation points $(\la_0, \la_1) = (0,1)$. 
	
	To check that $\hat{f}$ admits a discrete rotational symmetry, let $\tau\colon z\mapsto e^{i\pi/\mu_\alpha}z$ and recall that $\tau^*\Phi = R(\pi/\mu_\alpha)\Phi$ with $R(t) = C\exp(i t K)C\inv$. Compute $R_\alpha(\pi/\mu_\alpha) = -\I$ and deduce with \eqref{eq:simple-factor-h} that $\tau^*x=x$. 
	Moreover,  $\tau^*y=y$, so $\tau^*h = h$. We then have
	\begin{equation*}
		\tau^* \hat{\Phi} = \hat{R}\hat{\Phi}
	\end{equation*}
	with
	\begin{equation}\label{eq:Rhat}
		\hat{R} := \hat{g}R(\pi/\mu_\alpha)\hat{g}\inv.
	\end{equation}
	Note that $\hat{R} = \tilde{g}R(\pi/\mu_\alpha)\tilde{g}\inv$ where $\tilde{g}:=g_0\inv g\in\SL_2\Cc$ is holomorphic for $\la$ in a neighborhood of $\la_0=0$ and $\tilde{g}_0=\I$. 
	Moreover, the wrapping number of the catenoid is $2\mu_0$, so $R(\pi/\mu_\alpha)$ induces a rotation of angle $\frac{2\pi \mu_0}{\mu_\alpha}$. By proposition~\ref{prop:rigid-motion-homomorphism-E3}, $\hat{R}$ induces a rotation of the same angle. Therefore, a rotation of angle $2\pi/(2u\sqrt{p})$ in the domain induces a rotation of angle $2\pi/u$ in space.
	
	One can compute explicitly the location of the singularities of $h$ by solving for $z$ with the notations of \eqref{eq:simple-factor-h}:
	\begin{equation*}
		\Phi_\alpha\inv \ell = y \quad \iff\quad  \Phi_\alpha w \in \ell \quad \iff\quad -\frac{1}{\mu_0}\cdot \frac{(\mu_0+\mu_\alpha)z^{2\mu_\alpha} - (\mu_0-\mu_\alpha)}{(\mu_0-\mu_\alpha)z^{2\mu_\alpha}-(\mu_0+\mu_\alpha)} = \frac{\ell_1}{\ell_2}.
	\end{equation*}
	Considering $z^{2\mu_\alpha}$ as the unknown of the equation, it has exactly one solution. Therefore, the dressed surface admits one extra end on each fundamental domain of its rotational symmetry.
	Note that the change
	\begin{equation*}
		\vectorr{\ell_1}{\ell_2}\mapsto\vectorr{\omega \ell_1}{\ell_2},\quad \Phi\mapsto \matrix{\omega^{1/2}}{0}{0}{\omega^{-1/2}}\Phi\qquad (\omega\in\Ss^1),
	\end{equation*}
	induces a rotation of the surface, and justifies our normalization $\ell\in\RP^1$. This parameter will determine the location $z_1$ of the extra singularity in each fundamental domain. By \eqref{eq:metric-lwr}, this singularity is an end because 
	\begin{equation*}
		\det h \tendsto{z\to z_1}0\implies \norm{h\inv}\tendsto{z\to z_1}\infty \implies ds^2\tendsto{z\to z_1}\infty
	\end{equation*}
	and the end is planar because the Hopf differential holomorphically extends to $z_1$ by corollary~\ref{corollary:hopf-simple-factor-dressing}.
	
	The proof in $\Hh^3$ follows the same arguments with 
	\begin{equation*}
		u>0,\quad u\neq 1,\quad u\neq \frac{\sqrt{p+q}}{\sqrt{p}},\quad \ell\in\RP^1
	\end{equation*}
	so that $\alpha\notin\{0,1\}$. Let $\hat{g} := \sqrt{\det g_1}g_1\inv g$. Use proposition~\ref{prop:rigid-motion-homomorphism-H3} to show that the surface admits a discrete rotational symmetry.
\end{proof}

\begin{figure}[h]
	\centering
	\begin{subfigure}{0.49\textwidth}
		\includegraphics[width=\linewidth]{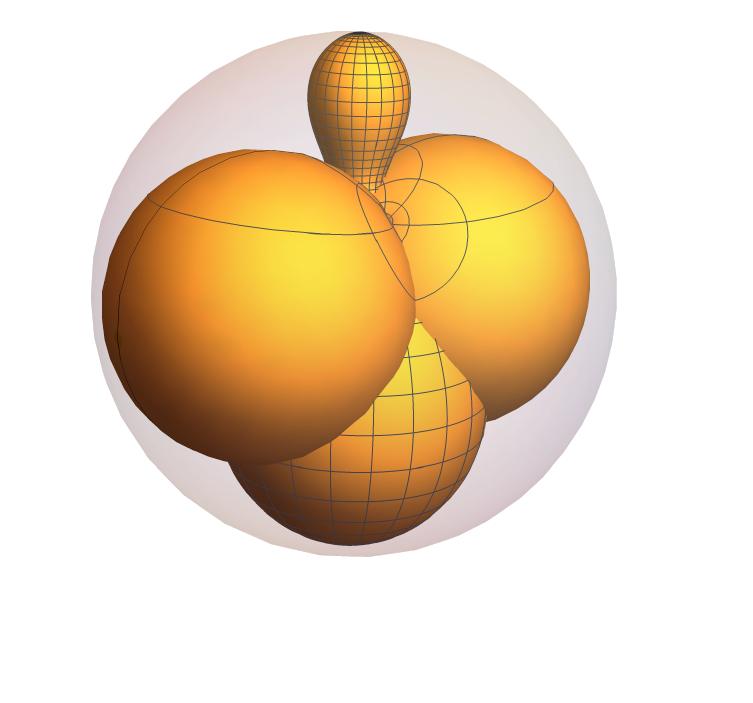}
		\caption{Full surface}
		\label{fig:fullH3}
	\end{subfigure}
	\hfill
	\begin{subfigure}{0.49\textwidth}
		\includegraphics[width=\linewidth]{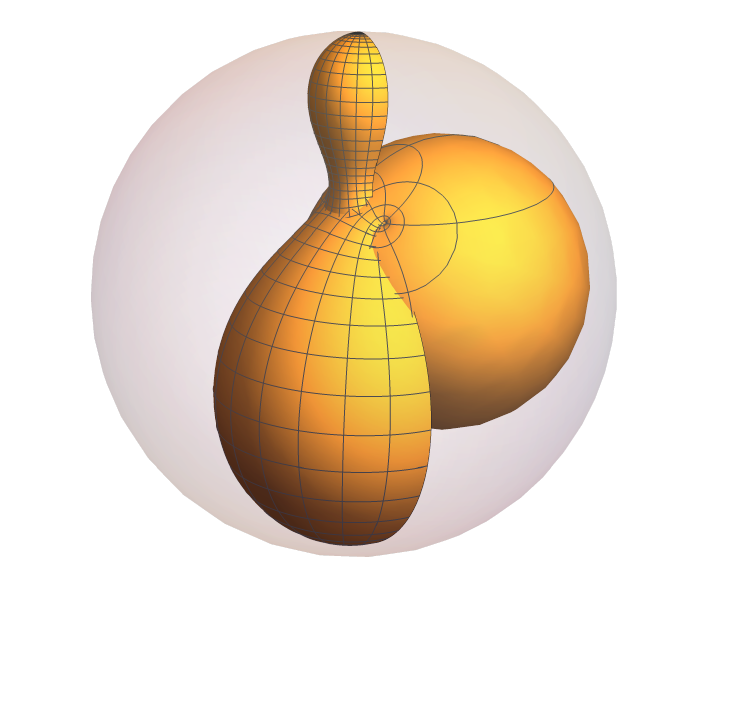}
		\caption{Half view}
		\label{fig:cutH3}
	\end{subfigure}
	\caption{Simple factor dressing of the catenoid in $\Hh^3$ viewed in the Poincaré ball model, as in example~\ref{ex:dressed-catenoid2} with $(p,q)=(\frac{1}{4},-0.1)$ and $(u,\ell) = (2,2)$.}
	\label{fig:dressed-catenoid2H3}
\end{figure}

\begin{example}\label{ex:dressed-catenoid2}
	Dressing of a singly-wrapped catenoid (figures~\ref{fig:dressed-catenoid2}--\ref{fig:dressed-catenoid2H3}).
	The original catenoid is given by
	\begin{equation*}
		p=\frac{1}{4}\quad \text{and}\quad q\in(-1/4,\infty)\backslash\{0\},
	\end{equation*}
	and the parameters for dressing are
	\begin{equation*}
		u \in\Zz_{\geq 2}\quad \text{and}\quad \ell\in\RP^1.
	\end{equation*}
	Then the dressed immersion $\hat{f}$ has a discrete rotational symmetry of order $u$ both in the domain and in space. In particular, it closes on $\Cc^*\backslash\{z_k\}_{k=1,\cdots,u}$ where the $z_k$ are the solutions of
	\begin{equation*}
		z_k^u = \frac{u(2\ell_2-\ell_1)-(2\ell_2+\ell_1)}{u(\ell_1-2\ell_2)-(2\ell_2+\ell_1)},\quad \ell := \spann(\ell_1,\ell_2).
	\end{equation*}
\end{example}

\begin{example}\label{ex:dressed-catenoid1plane}
	Dressing of a doubly-wrapped catenoid (figures~\ref{fig:dressed-catenoids-R3}--\ref{fig:dressed-catenoidR31plane}). The original catenoid is given by 
	\begin{equation*}
		p=1\quad \text{and}\quad q\in(-1,\infty)\backslash\{0\},
	\end{equation*}
	and the parameters for dressing are
	\begin{equation*}
		u = \frac{1}{2}\quad \text{and}\quad \ell\in\RP^1.
	\end{equation*}
	The fundamental piece is $\Cc^*\backslash\{z_1\}$ where
	\begin{equation*}
		z_1 = \frac{3\ell_1-\ell_2}{3\ell_2+\ell_1},\quad \ell := \spann(\ell_1,\ell_2) 
	\end{equation*}
	and the dressed immersion $\hat{f}$ closes on $\Cc^*\backslash\{z_1\}$. This construction is inspired by \cite{cho2022new}.
\end{example}

\begin{figure}[h]
	\centering
	\begin{subfigure}{0.49\textwidth}
		\includegraphics[width=\linewidth]{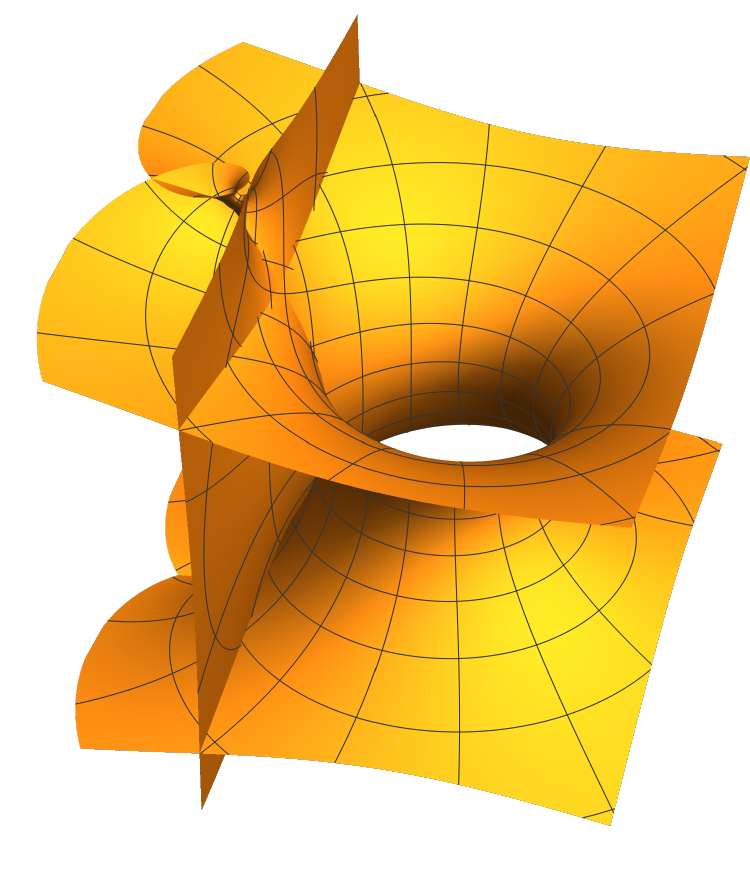}
		\caption{Full surface}
		\label{fig:fullR31plane}
	\end{subfigure}
	\hfill
	\begin{subfigure}{0.49\textwidth}
		\includegraphics[width=\linewidth]{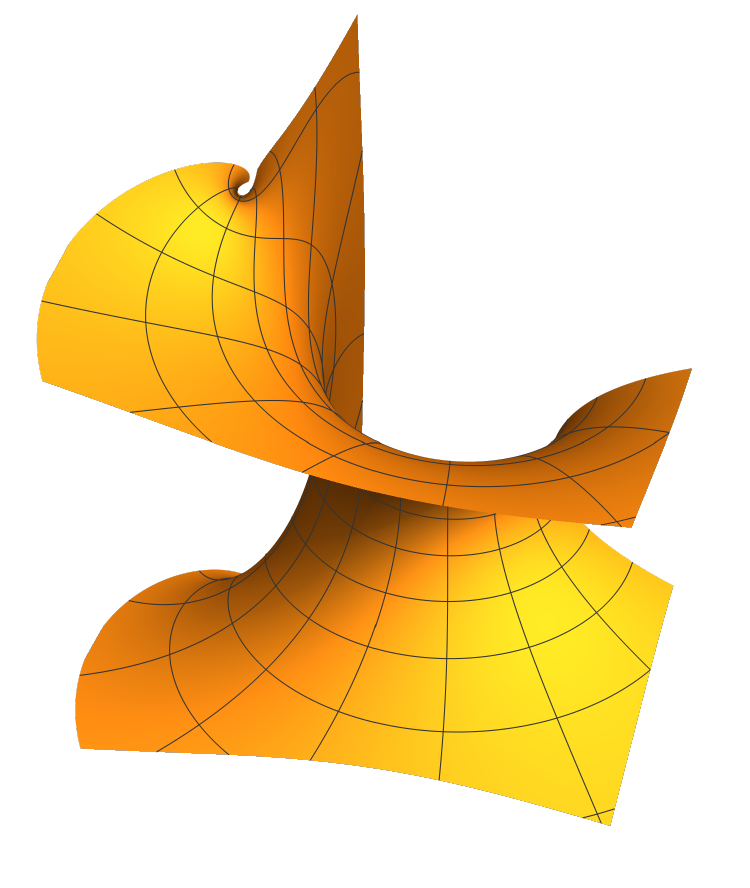}
		\caption{Half view}
		\label{fig:cutR31plane}
	\end{subfigure}
	\caption{Simple factor dressing of the doubly-wrapped catenoid in $\Ee^3$, as in example~\ref{ex:dressed-catenoid1plane} with $(p,q)=(1,1)$ and $(u,\ell)=(\frac{1}{2},1)$ (see also figure~\ref{fig:dressed-catenoids-R3}).}
	\label{fig:dressed-catenoidR31plane}
\end{figure}

\begin{remark}
	By computing explicitly the Gauss maps of the examples above, and with Corollary 1 of \cite{hertrich2017minimal}, one can show that the dressed catenoids are Darboux transformations of the standard catenoid.
\end{remark}

\bibliography{biblio}

\end{document}